\documentclass[aap,preprint]{imsart}

\RequirePackage[OT1]{fontenc}
\RequirePackage{amsthm,amsmath}
\RequirePackage[numbers]{natbib}

\usepackage{xifthen}
\usepackage{multirow,rotating,lscape}
\usepackage{subfig}

\usepackage{pkgfile}

\allowdisplaybreaks

\begin{document}

\begin{frontmatter}
\title{Collision Times in Multicolor Urn Models and Sequential Graph Coloring With Applications to Discrete Logarithms}
\runtitle{Collision Times in Multicolor Urn Models}

\begin{aug}
\author{\fnms{Bhaswar~B.} \snm{Bhattacharya}\ead[label=e1]{bhaswar@stanford.edu}}

\affiliation{Stanford University}

\address{Department of Statistics\\ Stanford University,\\ California, USA\\ \printead{e1}}

\end{aug}

\begin{abstract}
Consider an urn model where at each step one of $q$ colors is sampled according to some probability distribution and a ball of that color is placed in an urn. The distribution of assigning balls to urns may depend on the color of the ball. Collisions occur when a ball is placed in an urn which already contains a ball of different color. Equivalently, this can be viewed as sequentially coloring a complete $q$-partite graph wherein a collision corresponds to the appearance of a monochromatic edge. Using a Poisson embedding technique, the limiting distribution of the first collision time is determined and the possible limits are explicitly described. Joint distribution of successive collision times and multi-fold collision times are also derived. The results can be used to obtain the limiting distributions of running times in various birthday problem based algorithms for solving the discrete logarithm problem,  generalizing previous results which only consider expected running times. Asymptotic distributions of the time of appearance of a monochromatic edge are also obtained for other graphs.
\end{abstract}



\end{frontmatter}

\section{Introduction} Suppose the vertices of a finite graph $G = (V,E)$, with $|V|=N$, are colored independently and uniformly at random with $c$ colors. The probability that the resulting coloring has no monochromatic edge, that is, it is a proper coloring is $\chi_G(c)/c^N$, where $\chi_G(c)$ denotes the number of proper colorings of $G$ using $c$-colors. The function $\chi_G$ is the chromatic polynomial of $G$, which is a central object in graph theory \cite{chromaticbook,toft_book}. A natural extension is to consider a general coloring distribution $\vec p=(p_1, p_2, \ldots , p_c)$, where the probability that a vertex is colored with color $a\in [c]$ is $p_a$ which is independent of the colors of the other vertices, where $p_a\ge 0$, and $\sum_{a=1}^c p_a=1$. Then the probability that $G$ is properly colored is  related to Stanley's generalized chromatic polynomial \cite{fadnavis,stanley}. Limit theorems for the number of monochromatic edges under the uniform coloring distribution, that is, $p_a=1/c$ for all $a\in [c]$, was derived recently by Bhattacharya et al. \cite{bbbpdsm}.

When the underlying graph $G$ is a complete graph, this reduces to the well-known birthday problem: by replacing the colors by birthdays, occurring with possibly non-uniform probabilities, the birthday problem can be seen as coloring the vertices of a complete graph independently with $c=365$ colors. The event that two people share the same birthday is the event of having a monochromatic edge in the colored graph.
The birthday problem was generalized to the sequential setting by Camarri and Pitman \cite{camarripitmanejp} as follows: in a stream of people, determine the distribution of the first time that a person arrives whose birthday is the same as that of some person previously in the stream. More generally, they derived the asymptotic distribution of the first repeat time in an i.i.d. $\mathbb P_N$ sequence, in a limiting regime with the probability distribution $\mathbb P_N$ depending on a parameter $N\in \mathbb N$. Formally, suppose that the $N$-th distribution $\P_N$ is a {\it ranked discrete distribution},  $p_{N1}\geq p_{N2}\geq \ldots \geq 0 \text{ and } \sum_{i=1}^\infty p_{N i} = 1$. A 
sequence $\sX_N:=(X_{N1}, X_{N2}, \ldots)$ of i.i.d. random variables distributed as $\P_N$ is said to have a {\it repeat at time $t$}, if $X_{Nt}=X_{Ns}$, for some $s<t$. The {\it first repeat time} 
\begin{align}\label{RN1}
R_{N1}=\inf\{t\in \N: r_N(t)=1\},
\end{align}
where $r_N(t)$ is the number of repeats in the sequence $\sX_N$ up to time $t$. In other words, $R_{N1}$ is the first time some element is observed twice in the sequence $\sX_N$. More generally, the $m$-th repeat time $R_{Nm}$ of the sequence $\sX_N$, is the minimum $t$ such that $r_N(t)=m$, that is, the first time that $m$ repetitions occur in the sequence $\sX_N$.

This can also be viewed as sequentially coloring the vertices of the infinite complete graph, independently with probability $\P_N$, and $R_{N1}$ is the first time that a monochromatic edge appears. Another way to rephrase this is in terms of an urn model with urns (corresponding to birthdays) indexed by $\{1, 2, \ldots \}$ and with infinitely many balls. Initially all the urns are empty, and at every subsequent time step a ball is dropped into urn $i$ with probability $p_{N i}$, where $\sum_{i=1}^\infty p_{N i}=1$, and $p_{N1}\geq p_{N2}\cdots \geq 0$. Then $R_{N1}$ is the first time that there are two balls in the same urn.

In the {\it uniform} case, where $p_{N i}=1/N$ for all $i\in\{1, 2, \ldots, N\}$, it is well known that for all $r\geq 0$, $R_{N1}/\sqrt{N}$ converges to the Rayleigh distribution with parameter 1. Camarri and Pitman \cite{camarripitmanejp} used the Poisson embedding technique and characterized the set of all possible asymptotic distributions of $R_{N1}$ derived from any sequence of general ranked distributions. In the uniform case, Arratia et al. \cite{arratiabirthday} derived the limiting distribution of the $m$-th repeat time $R_{Nm}$, when $m=O(N)$. 

The non-sequential version of the urn model described above is the classical occupancy scheme with infinitely many boxes, where balls are thrown independently into boxes with probability $\P_N$. Asymptotics for the number of boxes occupied by exactly $r$ balls are well known \cite{barbour_gnedin_joint_normality,gnedinsurvey}. In a different context, Paninsky \cite{paninsky} used $B_1$, the number of boxes with 1 ball, for testing uniformity given sparsely-sampled discrete data.  The  Poisson embedding technique is also useful in other occupancy urn problems: Holst \cite{holstreview,holstgeneralbirthday} used it to derive moments of a general quota problem; Holst \cite{holstextremes}  and later Neal \cite{couponcollectorneal} also used these techniques to obtain limiting distributions in coupon-collector problems.  For other variations of occupancy urn models and their applications, refer to \cite{gnedinsurvey,johnsonkotzurnmodel} and the references therein. For embedding P\'olya-type urn schemes into continuous time Markov branching processes refer to \cite{athreyakarlin,johnsonkotzurnmodel,urnbook}.



\subsection{Collision Times in a Multicolor Urn Model}A natural generalization of the birthday problem is to consider coincidences among individuals of different types, that is, in a room occupied with an equal number of boys and girls, when can one expect a boy and girl to share the same birthday. This can be viewed as an urn model with two colors, where balls are colored independently with probability 1/2 and placed in the urns uniformly. The event of having a matching birthday is same as having an urn with balls of both the colors. This event is often referred to as a collision. For exact expressions of the number of collisions, factorial moments and other related problems, refer to Nakata \cite{nakata} and the references therein. The number of collisions between two discrete distributions was also used by Batu et al. \cite{batu}  for distributional property testing. Wendl \cite{wendl} studied a very related problem and referred to some applications in collisions of airborne planes, celestial objects, and transportations.

In this paper we consider the sequential version of this problem, for general urn selection distributions. 

\begin{defn}\label{defn:uniform}
Consider an urn model with balls of $q$ distinct colors (corresponding to types) indexed by $\{1, 2, \ldots, q\}$, and urns (corresponding to birthdays) indexed by $\{1, 2, \ldots \}$. Initially all the boxes are empty, and at every subsequent time point the following steps are executed: 

\begin{description}
\item[1]({\it Color Selection}). A color is $a\in [q]$ is chosen uniformly, that is, with probability $1/q$.

\item[2]({\it Urn Selection}). If the color chosen is $a\in[q]$, then a ball with color $a$ is dropped into urn $i$ with probability $p_{Ni}$, where $p_{N1}\geq p_{N2}\cdots \geq 0$ and $\sum_{i=1}^\infty p_{N i}=1$, is a ranked discrete distribution.
\end{description}
\end{defn}

Let $C_t$ be the color of the ball chosen at the $t$-th step, and $Z_{Nt}$ be the urn to which the ball is assigned. The urn model described above is said to have a {\it collision at time $t$}, if $Z_{Ns}=Z_{Nt}$ and $C_s\ne C_t$, for some $s<t$. In other words, a collision happens when a ball is dropped in an urn which already contains a ball with a different color.  Given the above process, define the {\it first collision time} $T_{N 1}$ {\it to be the first time that there exist two balls with different colors in the same urn.}

Using the Poisson embedding technique, the limiting distribution of $T_{N 1}$ can be obtained:

\begin{thm}For $N, i \in \mathbb N$, let 
\begin{equation}
s_{N}=\left(\sum_{i}p_{N i}^2\right)^{\frac{1}{2}}, \quad \psi_{N i}=\frac{p_{N i}}{s_N},
\label{eq:sNphiN}
\end{equation}
Suppose that $\lim_{N\rightarrow \infty}p_{N 1}=0$, and $\psi_i=\lim_{N\rightarrow \infty}\psi_{N i}$ exists, for $i \in \mathbb N$. Then 
\begin{align}\label{eq:qcolorsfirstrepeatth}
\lim_{N \rightarrow \infty}\P(s_{N}& T_{N 1}> r)\nonumber\\
=&e^{-\frac{1}{2}\left(\frac{q-1}{q}\right)r^2\cdot\left(1-\sum_{i}\psi_i^2\right)}\prod_ie^{-\left(\frac{q-1}{q}\right) \psi_{i} r}\left(q-(q-1)e^{-\psi_i\frac{r}{q}}\right).
\end{align}
Conversely, if there exist positive constants $c_{N} \rightarrow 0$ and $d_{N}$ such that the distribution of $c_{N}(T_{N1}-d_{N})$ has a non-degenerate weak limit as $N \rightarrow \infty$, then $p_{N1}\rightarrow 0$ and limits $\psi_{i}$ exist as before. So the weak limit is just a rescaling of that described in (\ref{eq:qcolorsfirstrepeatth}), with $c_{N}/s_{N} \rightarrow \alpha$ for some $0 < \alpha  < \infty$, and $c_{N}d_{N}\rightarrow 0$.
\label{th:bipartitefirstrepeat}
\end{thm}


If the process described above is continued after the first collision time $T_{N1}$, more collisions occur. Recall, that a collision corresponds to a ball being dropped in an urn which already contains a ball with a different color.

\begin{defn}\label{mthcollision}
For $m\geq 1$, let $T_{Nm}$ be the time of the $m$-th collision, that is, 
$$T_{Nm}=\inf\left\{t \in \N: \sum_{i=1}^t K_t=m\right\},$$
where $K_t$ is the indicator variable which is 1 if and only if the urn model described above has a collision at time $t$. 
\end{defn}

From the continuous time embedding of the process, the joint convergence of the collision times can be obtained.

\begin{thm}\label{th:bipartitejointrepeat} Suppose $\lim_{N\rightarrow \infty}p_{N1} = 0$, $\psi_i=\lim_{N\rightarrow \infty} \psi_{Ni}$ exists for each $i \in \mathbb N$, and $s_N$ as in~(\ref{eq:sNphiN}). Then there is the convergence of $m$-dimensional distributions
$$(s_NT_{N1},  s_NT_{N2}, \ldots, s_NT_{Nm}) \stackrel{\sD}\rightarrow (\eta_1, \eta_2, \ldots, \eta_m),$$ where $0 < \eta_1 < \eta_2 < \cdots$ are the arrival times of a process $\sM$, which is the superposition of independent point processes $B^*$, $B_1^{-L_1}, B_2^{-L_2}, \cdots$, where: 
\begin{itemize}
\item $B^*$ is a Poisson process on $[0, \infty)$ of rate $(1-\sum_i\psi_i^2)t\cdot\left(1-\frac{1}{q}\right)$ at time $t$.

\item For each $i \in \mathbb N$, $B_i$ is the superposition of $q$ independent Poisson processes $$\{B^1_{i}(t)\}_{t\geq 0}, \{B^2_{i}(t)\}_{t\geq 0}, \ldots, \{B^q_{i}(t)\}_{t\geq 0}$$ on $[0, \infty)$ of rate $\psi_i/q$. Finally, $B_i^{-L_i}$ is the process $B_i$ with its first $L_i:=B_i(T_i')$ points removed, where $T_i'$ is the last arrival time in $B_i$  before  $T_i=\inf\{t \geq 0: B^a_{i}(t)>0 \text{ and } B^b_{i}(t)>0 \text{ for some } ,  a\ne b\}$.

\end{itemize}
\end{thm}

The time $T_i'$ defined above is the last arrival time when all points of $B_i$ have the same color. Therefore, removing the first $L_i:=B_i(T_i')$ points ensures that any subsequent arrival in $B_i$ corresponds to a collision  in the urn labelled $i$. 

The urn model described in Definition~\ref{defn:uniform} can be generalized further by considering non-uniform color selection and letting the probability of selecting an urn to depend on the color selected. 

\begin{defn} \label{defn:nonuniform}
Consider an urn model with balls of $q$ distinct colors indexed by $\{1, 2, \ldots, q\}$, and urns indexed by $\{1, 2, \ldots \}$.  Initially all the boxes are empty, and at every time instance the following steps are executed:

\begin{description}
\item[1]({\it Non-uniform Color Selection}). A color is chosen with probability distribution $\pmb c=(c_1, c_2, \ldots c_q)$, that is, the probability of selecting the color $a\in [q]$ is $c_a$, where $c_a>0$ and $\sum_{a=1}^qc_a=1$.

\item[2]({\it Non-uniform Urn Selection}). If the color chosen is $a\in[q]$, then a ball with color $a$ is dropped into urn $i$ with probability $p_{Ni, a}$, where $\sum_{i=1}^\infty p_{Ni, a}=1$, for all $a\in [q]$.
\end{description}
\end{defn}

As in Definition~\ref{defn:uniform}, denote by $T_{N1}$ the first time there exist two balls with different colors in the same urn.  


\begin{thm}For $a\in [q]$, and $N, i \in \mathbb N$, let $$s^2_{N}=\sum_{i}\left(\sum_{a} c_ap_{Ni, a}\right)^2,\quad    \psi_{Ni, a}=\frac{p_{Ni, a}}{s_{N}}.$$ Suppose that $\lim_{N\rightarrow \infty}\max_ip_{Ni, a}=0$ and $\psi_{i, a}=\lim_{N \rightarrow \infty}\psi_{Ni, a}$ exists, for all $a\in [q]$ and $i \in \mathbb N$. Moreover, assume that $\phi_a=\lim_{N\rightarrow \infty}\sum_i\psi_{Ni, a}^2$ exists for all $a\in[q]$. Then 
\begin{align}
\lim_{N\rightarrow \infty}\P(s_{N}& T_{N1}> r)\nonumber\\
=& e^{-\left(1-\beta\right)\frac{r^2}{2}}\prod_ie^{-r\sum_{a=1}^q c_a\psi_{i, a}} \left(1+\sum_{a=1}^q e^{\psi_{i, a}c_a r}-q\right),
\label{eq:qcolorsgen}
\end{align}
where $\beta=\sum_{a=1}^q c_a^2\phi_a+\sum_{i}\sum_{a\ne b} c_ac_b\psi_{i, a}\psi_{i, b}$.
\label{th:qcolorurn}
\end{thm}

Theorem \ref{th:bipartitefirstrepeat} is a special case of the above theorem when $c_a=1/q$ and $p_{Ni, a}=p_{N i}$, for all $a\in [q]$. Another special case was considered by Selivanov \cite{selivanov}, where only Rayleigh distributions were obtained as limits.\footnote{Selivanov \cite[Theorem 4.1]{selivanov} claims that $s_NT_{N1}$ converges to a Rayleigh distribution, whenever $\sum_{i}p_{Ni}^2\rightarrow 0$ and $p_{N1}(\sum_{i}p_{Ni}^2)^{-\frac{1}{2}}<c$, for some constant $c$. However, the second condition is vacuously true for all distributions, for any $c>1$ (since $p_{Ni}^2\leq \sum_{i}p_{Ni}^2$, for all $i\geq 1$, implies that $p_{N1}^2\leq \sum_{i}p_{Ni}^2$). This implies that $s_NT_{N1}$ converges to a Rayleigh distribution, whenever $\sum_{i}p_{Ni}^2\rightarrow 0$. However, Theorem~\ref{th:bipartitefirstrepeat} shows that this is clearly incorrect, since the conditions $p_{N1}\rightarrow 0$ and $\sum_{i}p_{Ni}^2\rightarrow 0$ are equivalent (see Examples \ref{exnu1} and \ref{exnu2} for specific counterexamples).  A possible fix to Selivanov's condition is to assume that $p_{N1}(\sum_{i}p_{Ni}^2)^{-\frac{1}{2}}\rightarrow 0$. This would imply that $\psi_i=\lim_{N\rightarrow \infty}\psi_{Ni}=0$, for all $i \geq 1$, and by~\eqref{eq:qcolorsfirstrepeatth},  $\lim_{N \rightarrow \infty}\P(s_{N}T_{N 1}> r)=e^{-\frac{r^2}{4}}$, for $q=2$.}  Recently, Galbraith and Holmes \cite{birthday_discrete_logarithm} considered a variant of the urn model in Definition~\ref{defn:nonuniform}, where the color selection probabilities change with time, and used the Chen-Stein method to determine the expected first collision time. Theorem \ref{th:qcolorurn} extends these results and characterizes the different limiting distributions that may arise. Moreover, this general theorem can be used to find the asymptotic distributions of the running times for a   class of algorithms for solving the discrete logarithm problem (DLP) that requires generalizations of the birthday problem (details in Section \ref{sec:introdlp} and Section \ref{sec:dlp}). 

\subsection{Sequential Graph Coloring}
\label{sec:paintro}

The repeat time $R_{N1}$ of Pitman and Camarri \cite{camarripitmanejp} is the first time
when a monochromatic edge appears while sequentially coloring the vertices of the (infinite) complete graph, independently with probability distribution $\P_N$.  The collision time $T_{N1}$ in the urn model defined in the previous section is the first time a monochromatic edge appears while sequentially coloring the (infinite) complete $q$-partite graph, where at every step one of the $q$ partite sets is chosen uniformly at random, and a vertex in that set is colored independently with probability $\P_N$. Similar questions can be asked for any sequence of naturally growing graphs, which motivate us to formulate the following general problem.

Let $\sG:=(G_{t})_{t\geq 1}$ be a deterministic sequence graphs $G_{t}=(V_{t}, E_{t})$, with $V_{t+1}\subset V_{t}$, $|V_{t+1}|=|V_{t}|+1$, and $E_{t}\subseteq E_{t+1}$. For $N\geq 1$, consider the following sequential coloring scheme:\footnote{More formally, we have a triangular array of growing graphs $((G_{Ns}))_{s\geq 1}$, whose vertices are colored independently with probability distribution $\P_N$. For notational
simplicity the process is only described for a sequence of growing graphs $(G_t)_{t\geq 1}$, and in all the examples considered this simplification suffices.}
\begin{itemize}

\item  Every vertex in $V_{1}$ is colored independently with a ranked discrete probability distribution $\P_N$. 

\item For $t\geq 2$, the new vertex $v\in V_{t}\backslash V_{t-1}$ is colored with $\P_N$:
$$\P(\text{the vertex } v \text{ has color } i\in \N)=p_{Ni},$$
independent of the color all the other vertices.\footnote{This should not be confused with the color of the ball in the urn model, described in the previous section. The urn model corresponds to coloring a complete $q$-partite graph, and the color of a ball corresponds to which of the $q$ sets the vertex belongs to.} Define the first {\it collision time} $T_{N1}^\sG$ to be the first index $s$ when a monochromatic edge $(u, v)\in E_s$ appears.  
\end{itemize}
Note that this general framework includes the repeat time $R_{N1}$ defined in~\eqref{RN1} (take $G_t=K_t$ the complete graph on $t$ vertices), and the collision time $T_{N1}$ ($G_t$ is a complete $q$-partite graph on $t$ vertices with the added randomness that at every step one of the $q$ partite sets is chosen uniformly at random).

A popular model for evolving random graphs is the {\it preferential attachment} (PA) model, introduced in a seminal paper by Barab\'asi and Albert \cite{pamodel}. It builds on the paradigm that new vertices are attached to those already present with probability proportional to their degree. This model enjoys many properties observed in social networks and other real world networks: the power law distribution of vertex degrees, a small diameter, and a small average degree \cite{bollobaspa,bollobasdiam}. For every fixed integer $m \geq  2$, the $PA(m)$ model is formally defined as follows:  
\begin{enumerate}
\item[]The graph sequence grows one vertex at a time, and at the $t$-th step the graph  $G_m^t$ is an undirected graph on the vertex set $V := [t]$ defined inductively as follows. $G^1_t$ consists of a single vertex with $m$ self-loops. For all $t > 1$, $G^t_m$ is built from $G^{t-1}_m$ by adding a new node labelled $t$ together with $m$ edges $e^t_1 = (t, v_1),\ldots , e^t_m=(t, v_m)$ inserted one after the other in this order. Let $G^t_{m,i-1}$ denote the graph right before the edge $e_i^t$ is added. Let $M_i = \sum_{v\in V} d_{G^t_{m,i-1}}(v)$ be the sum of the degrees of all the nodes in $G^t _{m,i-1}$. The endpoint $v_i$ is selected randomly such that $v_i = u$ with probability $d_{G^t_{m,i-1}}(u)/(M_i + 1)$, except for $t$ that is selected with probability $(d_{G^t_{m,i-1}}(t) + 1)/(M_i + 1)$. 
\end{enumerate}
Note that the graph $G_m^t$ can have loops and multiple edges. However, it forms a vanishing fraction of the total edges and for the coloring problem it suffices to consider the underlying simple graph (to be denoted by $S(G_{m}^t)$). 

By taking $\sG=(S(G_{m}^t))_{t\geq 1}$,  define  $T_{N1}^{PA(m)}$ to be the first time there is a monochromatic edge in the sequential coloring of $\sG$. Using the Stein's method for Poisson approximation the following theorem can be proved:

\begin{thm}Let $s_{N}$ be as in~\eqref{eq:sNphiN} and $\lim_{N\rightarrow \infty}p_{N1}\rightarrow 0$, as $N\rightarrow \infty$. Then $s_{N}^2T_{N1}^{PA(m)}\stackrel{\sD}\rightarrow \mathrm{Exp}(m)$, the exponential distribution with parameter $m$.
\label{th:palimitdistribution}
\end{thm}

The asymptotics for collision times can also be studied for any deterministic sequence of graphs which grow naturally one vertex at a time. This is demonstrated for the infinite path: take $\cZ=(P_t)_{t\geq 1}$, where $P_t$ is the path with $t$ vertices, and define $T_{N, m}^\cZ$ to be the first time there exists a monochromatic path with $m$ vertices while sequentially coloring $\cZ$. As in Theorem~\ref{th:palimitdistribution}, the limit distribution of $T_{N, m}^\cZ$ can be proved (see Theorem \ref{th:pathlimitdistribution}). 


\subsection{Applications to the Discrete Logarithm Problem}
\label{sec:introdlp}

The discrete logarithm problem (DLP) in a finite group $G$ is as follows: given $g, h \in G$ find an integer $a$ such that $h = g^a$.  Due to its presumed computational difficulty, the problem figures prominently in various cryptosystems, including the Diffie-Hellman key exchange, El Gamal system and elliptic curve cryptosystems. The best algorithms to solve the discrete logarithm problem in a general group originate in the seminal work of Pollard \cite{pollardmoc,pollardkangaroo}. A standard variant of the classical Pollard Rho algorithm for finding discrete logarithms can be described using a Markov chain on the cycle. The running time of the algorithm is the collision time of the Markov chain that is, the first time the chain visits a state that was already visited. Several years later, Kim et al. \cite{peresbirthday} finally proved the widely believed $\Theta(\sqrt {|G|})$ collision time for this walk.

The discrete logarithm problem in an interval asks: given $g, h \in G$ and an integer $N$ find an integer $a$, if it exists, such that $h = g^a$ and $0 \leq  a < N$. The DLP in an interval can be solved using the baby-step-giant-step algorithm in $\sqrt 2N$ group operations and storage of $O(\sqrt N)$ group elements. The Pollard kangaroo method \cite{pollardkangaroo} was designed to solve the DLP in an interval using a constant number of group elements of storage. Using distinguished points a heuristic average case expected complexity of essentially $2\sqrt N$ group operations and low storage can be obtained. Montenegro and Tetali \cite{montenegrotetali} gave a more rigorous analysis of the kangaroo method. Recently, Galbraith et al. \cite{galbraithpollardruprai} used a 4-kangaroo method, instead of the usual two, to obtain a heuristic average case expected running time of $(1.715+o(1))\sqrt{N}$.

There has been several recent work extending and improving Pollard's algorithms for variants of the discrete logarithm problem which require generalizations of the birthday problem \cite{galbraithpollardruprai,galbraithpkc,galbraithmultidimensional}. Gaudry and Schost \cite{gaudryschost} presented one of the first birthday problem based methods for solving the DLP in an interval. The algorithm is based on the collision time of 2 independent pseudo-random walks. A {\it tame walk} is a sequence of points $\{g^{a_i}\}_{i\geq 1}$ where $a_i\in T$ and a {\it wild walk} is a sequence of points $g^{b_i} = hg^{a_i}$ with $b_i\in W$, where $T, W\subseteq \{1, 2, \ldots, |G|\}$ are the tame and wild sets, respectively.  When the same element is visited by two different types of walk, there is a tame-wild collision giving an equation of the form $g^{a_i} = hg^{b_j}$, and the DLP is solved as $h = g^{a_i-b_j}$. Therefore, the running time of the algorithm is the time required until a tame and wild walk collide. 
The average expected running time of this algorithm is $2.08\sqrt N$ group operations on a serial computer. Recently, Galbraith et al. \cite{galbraithpollardruprai} proposed a four-set Gaudry-Schost algorithm with heuristic average case expected running time of $(1.661+o(1))\sqrt{N}$. Later, modifying the Gaudry-Schost algorithm and using a variant of the birthday paradox, Galbraith and Ruprai \cite{galbraithpkc} proposed an improvement in groups in which inversions are faster than general group operations, as in elliptic curves. This algorithm will be referred to as the accelerated Gaudry-Schost algorithm, and has a heuristic average case expected running time of approximately $1.36\sqrt N$ group operations.

In the analyses of all such algorithms the quantity used to compare the running times is the expectation of the tame-wild collision time, averaged over all problem instances. However, in the light of the above theorems, the asymptotic distribution of the running time of all such algorithms can be obtained, under the assumption that the pseudorandom walks performed by the algorithms are sufficiently random and their running times can be analyzed by an idealized birthday problem involving the tame-wild collision. This is derived for the  Gaudry-Schost algorithm (Theorem \ref{th:gaudryschost}) and the accelerated Gaudry-Schost algorithm of Galbraith and Ruprai (Theorem \ref{th:dlpdistribution}). To the best of our knowledge, these are the first known results about the limiting distributions of the running times of these algorithms. Though these results are based on some heuristic assumptions, they give considerable insight about the dependence between the running times and the complexity of the problem instance. 

\subsection{Organization of the Paper} The paper is organized as follows: 
The proofs of Theorem \ref{th:bipartitefirstrepeat} and Theorem \ref{th:bipartitejointrepeat}, and examples are given in Section \ref{sec:limitcollision}. An  analogous limit theorem for the $m$-fold collision time is proved in Section \ref{sec:mfold}. In Section \ref{sec:qcolor} the generalized urn model is considered and the proof of Theorem \ref{th:qcolorurn} is presented. The asymptotic distributions of the running times of algorithms for the discrete logarithm problem are proved in Section \ref{sec:dlp}. The limiting distributions of the collision times for the preferential attachment model and the infinite path are derived in Section \ref{sec:steinlimit}. 

\section{Proofs of Theorems \ref{th:bipartitefirstrepeat} and \ref{th:bipartitejointrepeat}}
\label{sec:limitcollision}

In this section limiting distributions of collision times in the urn model described in Definition~\ref{defn:uniform} are derived. 

\subsection{Proof of Theorem \ref{th:bipartitefirstrepeat}} Let $\cP$ be a homogeneous Poisson process on $\cR:=[0, \infty )\times[0, 1]$ of rate 1 per unit area, with points $\{(S_1, W_1),$ $(S_2, W_2), \ldots \}$, where $0 < S_1 < S_2 < $ are the points of a homogeneous Poisson process on $[0, \infty)$ of rate 1 per unit length, and $W_1, W_2, \ldots $ are i.i.d. $\dU(0, 1)$. Let $\cR_t=[0, t]\times[0, 1]$ and $\cP(t)$ be the restriction of $\cP$ to $\cR_t$.


\begin{itemize}

\item Color the points in $\cP$ independently with one of $q$ colors, $\{1, 2, \ldots, q\}$ with probability $1/q$, that is, $$\P((S_i, W_i)\in \cP \text{ has color } a\in [q])= 1/q,$$ independently for every point in $\cP$. For $a\in [q]$ denote by $\cP^a$ the subsets of $\cP$ colored $a\in [q]$. By the marking theorem \cite{poisson}, $\cP^1, \cP^2, \ldots, \cP^q$ are independent Poisson process  each with rate $1/q$ on $\cR$.
 
\item  For $N\geq 1$, partition $[0, 1]$ into intervals $J_{N 1}, J_{N 2}, \ldots, $ such that the length of $J_{N i}$ is $p_{N i}$  (see Figure~\ref{fig:embedding}). For $t\geq 0$ and $a\in [q]$, let 
\begin{equation}
\cP_{N i}=\cP\cap [0, \infty)\times J_{N i} \quad \text{and} \quad \cP^a_{N i}=\cP^a\cap [0, \infty)\times J_{N i}.
\label{eq:colorstrip}
\end{equation}
Clearly, $\cP_{N 1}, \cP_{N 2}, \ldots$ are independent Poisson processes with rates $p_{N 1}, p_{N 2}, \ldots$, respectively; and for $a\in [q]$, $\cP_{N 1}^a, \cP_{N 2}^a, \ldots$ are independent Poisson processes with rates $p_{N 1}/q, p_{N 2}/q, \ldots$, respectively.

\end{itemize}

\begin{figure*}[h]
\centering
\begin{minipage}[c]{1.0\textwidth}
\centering
\includegraphics[width=4.25in]
    {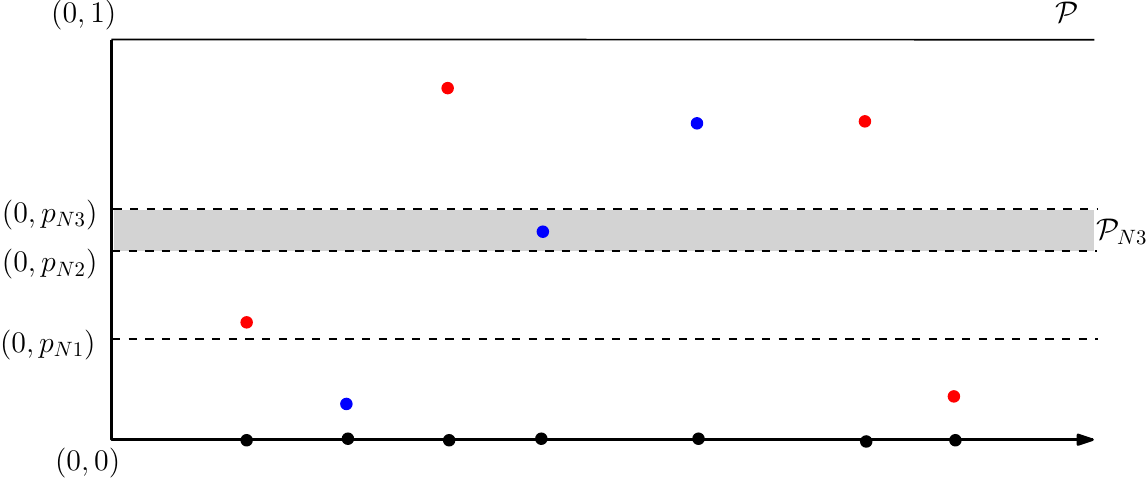}\\
\end{minipage}
\caption{\small{A schematic of the Poisson embedding for $q=2$ colors. Points are colored red or blue  with probability $\frac{1}{2}$.}}
\label{fig:embedding}
\end{figure*}


The collision time $T_{Nm}$ (Definition \ref{mthcollision}) can be described in terms of the above process: let $C_{j}$ be the color of the point $(S_j, W_j)$ and $Z_{Nj}=\sum_{i} i \pmb 1\{ W_j\in J_{Ni}\}$. The sequence $\{(C_j, Z_{Nj})\}_{j \geq 1}$ in the discrete time model corresponds to the color of the $j$-th ball and the urn to which the $j$-th ball is assigned. For $j \in \N$,  define $$K_j:=\pmb 1\{\exists n \in \N \text{ with } Z_{Nj}=Z_{Nj'}=n \text{ and } C_{j'}\ne C_{j}, \text{ for some } j'<j\},$$ the indicator that there is a collision at the $j$-th step. The $m$-th collision time (recall Definition~\ref{mthcollision}) is defined as
\begin{equation}
T_{Nm}\stackrel{D}=\inf\left\{j \in \N: \sum_{i=1}^j K_i=m\right\}.
\label{eq:TNm}
\end{equation} 
In particular, the first collision time $T_{N1}$ (defined as in Definition~\ref{defn:uniform}) is 
$$\inf\{j \in \N: \exists n \in \N \text{ with } Z_{Nj}=Z_{Nj'}=n \text{ and } C_{j'}\ne C_{j}, \text{ for some } j'<j \}.$$

\begin{lem}Let $\tau_{N m}=\inf \{t\in \R: |\cP(t)|\geq T_{N m}\}$, where $T_{N m}$ is as defined in~(\ref{eq:TNm}). Then $\frac{T_{N m}}{\tau_{N m}}\pto 1$, whenever $p_{N 1}\rightarrow 0$, as $N \rightarrow \infty$.
\label{lm:continuoustodiscrete}
\end{lem}

\begin{proof} By the strong law of large numbers $|\cP(t)|/t$ converges almost surely to 1 as $t \rightarrow \infty$. Therefore, it suffices to show $\tau_{N 1}$ converges in probability to infinity as $N \rightarrow \infty$, since by definition $|\cP(\tau_{N 1})|=T_{N 1}$. This implies $\tau_{N m}$ converges in probability to infinity, as $\tau_{N m}\geq \tau_{N 1}$, for $m \geq 1$.

By definition $\tau_{N 1}$ is 
$$\inf \{t\geq 0: \exists j\in \mathbb N \text{ with } |\cP_{N j}^\alpha(t)|>0, ~|\cP_{N j}^\beta(t)|>0, \text{ for some } \alpha\ne \beta \in [q]\},$$
where $ \cP_{N j}^a(t)$ is the restriction of $\cP_{N j}^a$ (defined in~(\ref{eq:colorstrip})) to $\cR_t$, for $a\in [q]$. This implies that 
\begin{eqnarray}
\P(\tau_{N 1}> t)&=&\prod_i \left(q\left(1-e^{-p_{N i}\frac{t}{q}}\right)e^{-(q-1)p_{N i}\frac{t}{q}}+e^{-p_{N i}t}\right)\nonumber\\
&=&e^{-\left(\frac{q-1}{q}\right) t}\prod_i \left(q-(q-1)e^{-p_{N i}\frac{t}{q}}\right),
\label{eq:lmcontinuous}
\label{eq:tN1}
\end{eqnarray}
Using $\log\left(q-(q-1)e^{-\frac{x}{q}}\right)\geq  \frac{q-1}{q}\left(x-\frac{x^2}{2}\right)$, for $x\geq 0$, (\ref{eq:lmcontinuous}) simplifies to 
$$|\log \P(\tau_{N 1}> t)|\leq \frac{1}{2}\left(\frac{q-1}{q}\right)t^2\sum_{i}p_{N i}^2\leq \frac{1}{2}\left(\frac{q-1}{q}\right)t^2 p_{N 1}\sum_{i}p_{N i}\rightarrow 0,$$
and the result follows. 
\end{proof}

Let $p_N=(p_{N 1}, p_{N 2}, \ldots)$ be the vector of probabilities. By the above lemma, to get the limiting distribution of $T_{N 1}$ it suffices to derive the limiting distribution of $\tau_{N 1}$. From~(\ref{eq:lmcontinuous}), $\log(\P(\tau_{N 1}>t))=g(t, p_N)$, where 
\begin{equation}
g(r, \vec \psi):=\sum_i \log\left\{e^{-\left(\frac{q-1}{q}\right) r \psi_i} \left(q-(q-1)e^{-\psi_{i}\frac{r}{q}}\right)\right\},
\label{eq:gdefine}
\end{equation}
for $r \geq 0$ and a vector $\vec \psi=(\psi_1, \psi_2, \ldots)$.

\begin{lem}Let $\vec \psi=(\psi_1, \psi_2, \ldots)$ and $\psi_1\geq \psi_2\geq \ldots \geq 0$ and $\sum_{i}\psi_i^2<\infty$. Then there exists a constant $0<c(q)\leq 1$, depending only on $q$, such that for $r\in R:=[0, c(q)/\psi_1)$, there exists 
 $\{a_s\}_{s=3}^\infty$ non-negative constants with
\begin{equation}
g(r, \vec \psi)=-\frac{1}{2}\left(\frac{q-1}{q}\right) r^2\sum_{i=1}^\infty\psi_{i}^2+\sum_{s\geq 3} (-1)^{s+1}a_sr^s\sum_{i=1}^\infty \psi_i^s,
\label{eq:gseries}
\end{equation}
and the above series is absolutely convergent.
\label{lm:series}
\end{lem}

\begin{proof}For $c(q)$ chosen small enough, $|1-e^{-\psi_1 r}|< 1/(q-1)$, for $r\in R$.  As $\psi_i\leq \psi_1$, $|1-e^{-\psi_i r}|< 1/(q-1)$, for $r\in R$ and all $i \in \mathbb N$. Now, using the expansion of $\log(1+z)$, for $|z|<1$,
\begin{align}
g(r, &\vec \psi)\nonumber\\
=&\sum_i\left\{-\left(\frac{q-1}{q}\right)r \psi_i + \sum_{s=1}^\infty\frac{(-1)^{s+1}(q-1)^s}{s}(1-e^{-\psi_{i}\frac{r}{q}})^s\right\}\nonumber\\
=&\sum_i\left\{ -\left(\frac{q-1}{q}\right)r \psi_i+ \sum_{s=1}^\infty\frac{(-1)^{s+1}(q-1)^s}{s}\left(\sum_{x=1}^\infty(-1)^{x+1}\frac{\psi_i^xr^x}{q^x x!}\right)^s\right\}\nonumber\\
=& T_1+T_2,
\label{eq:series}
\end{align}
where $$T_1=\sum_i (q-1) \left(\sum_{x=2}^\infty(-1)^{x+1}\frac{\psi_i^xr^x}{q^x x!}\right),$$ and $$T_2=
\sum_i \sum_{s=2}^\infty\frac{(-1)^{s+1}(q-1)^s}{s}\left(\sum_{x=1}^\infty(-1)^{x+1}\frac{\psi_i^xr^x}{q^x x!}\right)^s.$$

Define $$\sS=\sum_i  \left\{ (q-1) \left(\sum_{x=2}^\infty\frac{\psi_i^xr^x}{q^x x!}\right)+
\sum_{s=2}^\infty \sum_{\gamma_1, \ldots, \gamma_s\geq 1}\frac{(q-1)^s}{s}\frac{(\psi_ir)^{\sum_{b=1}^s\gamma_b}}{q^{\sum_{b=1}^s\gamma_b} \prod_{b=1}^s\gamma_b!}\right\}.$$
To show that (\ref{eq:series}) is absolutely convergent, it suffices to show $\sS< \infty$, whenever $r\in R$ and $C:=\sum_{i}\psi_i^2<\infty$. Let $\lambda=\frac{r\psi_1}{q}$, and observe, 
\begin{eqnarray}
\sS&\leq&C(q-1)\sum_{x=2}^\infty\frac{\psi_1^{x-2}r^x}{q^x}+
\frac{C}{\psi_1^2}\sum_{s=2}^\infty (q-1)^s\sum_{\gamma_1, \ldots, \gamma_s\geq 1}\left(\frac{r\psi_1}{q}\right)^{\sum_{b=1}^s\gamma_b}\nonumber\\
&\leq&\frac{C(q-1)r^2}{q^2}\frac{1}{1-\lambda}+
\frac{C}{\psi_1^2}\sum_{s=2}^\infty (q-1)^s\left(\sum_{x=1}^\infty \lambda^x\right)^s\nonumber\\
&\leq&\frac{C(q-1)r^2}{q^2}\frac{1}{1-\lambda}+
\frac{C}{\psi_1^2}\sum_{s=0}^\infty \left(\frac{\lambda(q-1)}{1-\lambda}\right)^s< \infty,
\label{eq:seriesabsconvergent}
\end{eqnarray}
when $r\in R$.
Therefore, by expanding (\ref{eq:series}) further and interchanging the order of the summation using the absolute convergence,~(\ref{eq:gseries}) follows.
\end{proof}


\subsubsection{Completing the Proof of~\eqref{eq:qcolorsfirstrepeatth} in Theorem \ref{th:bipartitefirstrepeat}}
Let $s_{N}$ and $\psi_{N i}$ be as defined in the statement of the theorem. By Lemma 
\ref{lm:continuoustodiscrete} it suffices to obtain the limiting distribution of $\tau_{N 1}$. From (\ref{eq:lmcontinuous}) it follows
\begin{equation}
\label{prob}
\P(s_N\tau_{N 1}> r)=\prod_i e^{-\left(\frac{q-1}{q}\right) \psi_{N i} r} \left(q-(q-1)e^{-\psi_{N i}\frac{r}{q}}\right).
\end{equation}
As $\sum_{i}\psi_{N i}^2=1$ and $\lim_{N\rightarrow \infty}\psi_{N i}=\psi_i$ exists for all $i$, by Fatou's lemma, $\sum_{i}\psi_i^2<\infty$. This implies that $\lim_{i\rightarrow \infty}\psi_i=0$. Therefore, for every fixed $r>0$ and there exists $j(r), N(r)$ be such that for $N> N(r)$, $\psi_{N {j(r)}}<c(q)/r$.  Lemma \ref{lm:series} then implies
\begin{eqnarray}
\cT_N&:=&\sum_{i> j(r)} \log\left\{e^{-\left(\frac{q-1}{q}\right) \psi_{N i} r} \left(q-(q-1)e^{-\psi_{N i}\frac{r}{q}}\right)\right\}\nonumber\\
&=&-\frac{1}{2}\left(\frac{q-1}{q}\right)r^2\sum_{i> j(r)}\psi_{N i}^2+\sum_{s=3}^\infty(-1)^{s+1}a_{s}r^s\sum_{i> j(r)} \psi_{N i}^{s},
\label{eq:geqfinite}
\end{eqnarray}
where $\{a_{s}\}_{s\geq 1}$ are non-negative constants. Note that $\lim_{N \rightarrow \infty}\sum_{i > j(r)}\psi_{N i}^2=1-\sum_{i \leq j(r)}\psi_{i}^2$. Moreover, for any $s\geq 3$ and $i > j(r)$, $\psi_{N i}^{s}\leq \psi_{N {j(r)}}^{s-2}\psi_{N i}^2$ and $\sum_{i}\psi_{N i}^2=1$. Therefore, taking limit in~(\ref{eq:geqfinite}) as $N\rightarrow \infty$,
\begin{equation}
\cT_N\rightarrow-\frac{1}{2}\left(\frac{q-1}{q}\right)r^2\left(1-\sum_{i=1}^{j(r)}\psi_{i}^2\right)+ \sum_{i> j(r)}\log\left(q-(q-1)e^{-\psi_{i}\frac{r}{q}}\right).
\label{eq:geq}
\end{equation}

Moreover, as $N\rightarrow \infty$,
$$\prod_{i=1}^{j(r)}e^{-\left(\frac{q-1}{q}\right) \psi_{N i} r} \left(q-(q-1)e^{-\psi_{N i}\frac{r}{q}}\right)\rightarrow \prod_{i=1}^{j(r)} e^{-\left(\frac{q-1}{q}\right) \psi_{i} r}\left(q-(q-1)e^{-\psi_{i}\frac{r}{q}}\right),$$
which combined with (\ref{eq:geq}) gives~(\ref{eq:qcolorsfirstrepeatth}).

\subsubsection{Proof of Converse in Theorem \ref{th:bipartitefirstrepeat}}

The converse to~\eqref{eq:qcolorsfirstrepeatth} is proved using the convergence of types, and the following lemma:


\begin{lem}\label{h}
Let $\alpha> 0$ and $\vec\psi:=(\psi_i, i\geq 1)$ a non-increasing sequence of reals with $\sum_{i=1}^N \psi_i^2\leq 1$. Then $(\alpha, \vec \psi)$ can be uniquely reconstructed from the function $r\rightarrow h(\alpha r, \vec\psi)$ for $r\in [0, \infty)$, where
\begin{align}
\label{hr}
h(r, \vec \psi):= e^{-\frac{1}{2}\left(\frac{q-1}{q}\right)r^2\cdot\left(1-\sum_{i}\psi_{i}^2\right)}\prod_ie^{-\left(\frac{q-1}{q}\right) \psi_{i} r}\left(q-(q-1)e^{-\psi_{i}\frac{r}{q}}\right).
\end{align}
\end{lem}

\begin{proof}Using \eqref{eq:series},  the sequence $\alpha,~\sum_{i=1}^\infty \psi_i^3, ~\sum_{i=1}^\infty \psi_i^4, \ldots, $ can be uniquely extracted from the function $r\rightarrow h(\alpha r, \vec\psi)$. Now, let $(J_i, i\geq 0)$ be a partition of the unit interval such that the length of $J_0$ is $1-\sum_{i=1}^\infty \psi_i^2$ and the length of $J_i$ is $\psi_i^2$, for all $i\geq 1$. Define $Z :=\sum_{i=1}^\infty \psi_i \pmb 1\{ U \in J_i\}$, where $U$ is a uniform $[0, 1]$ random variable. Then $\E(Z^k)=\sum_{i=1}^\infty \psi^{k+2}$, and these moments of $Z$ uniquely determine the distribution of $Z$ on $[0, 1]$.  Finally, it is easily seen that this distribution uniquely determines the  sequence $(\psi_1, \psi_2, \ldots )$.
\end{proof}

The converse to~\eqref{eq:qcolorsfirstrepeatth} now follows using the above lemma, and by taking sub-sequential limits and an application of convergence of types (Theorem 14.2, Billingsley \cite{billingsley}). 

%


\subsection{Proof of Theorem \ref{th:bipartitejointrepeat}}
\label{sec:pfjointdist}

Let $\cP$ be a homogeneous Poisson process on  $[0, \infty)\times [0, 1]$ of rate 1 per unit area, and $\cP_{Ni}$, $\cP_{Ni}^a$ be as defined before, for $i \in \N$ and $a\in [q]$. Note that the process $\cP_{Ni}$ is a Poisson process of rate $p_{N i}$, which is the superposition of $q$ independent Poisson processes $\cP_{Ni}^1, \cP_{Ni}^2 \ldots, \cP_{Ni}^q$ each of rate $p_{N i}/q$. 

Define $$F_{Ni}=\inf\{t \geq 0: |\cP_{Ni}^a(t)|> 0 \text{ and } |\cP_{Ni}^b(t)|> 0, \text{ for some } a\ne b \in [q] \}.$$
Note that the process $\cP_{Ni}$ has points $(S_1, W_1), (S_2, W_2), \ldots$, where the inter-arrival times $S_1, S_2-S_1, \ldots$ have independent exponential distribution with mean $1/p_{N i}$, and $W_1, W_2, \ldots$ are i.i.d. $\dU(0, 1)$. Every point of $\cP_{Ni}$ is colored by a color $a\in [q]$ with probability $1/q$, and the set of points colored $a$  is the process $\cP_{Ni}^a$. Let $\cP_{Ni}^{-L_{Ni}}$ be the process $\cP_{Ni}$ obtained removing the first $L_{Ni}$ points, where $L_{Ni}=|\cP_{Ni}(F_{Ni}')|$  and  $F_{Ni}^{'}$ is the last arrival time in $\cP_{Ni}$ before $F_{Ni}$, that is, the last arrival time when all points in $\cP_{Ni}$ are marked with the same color. Note that $F_{Ni}'$ is distributed as $\sum_{j=1}^W W_{j}$, where $W_j$ are i.i.d. exponential with mean $1/p_{N i}$ and $W$ is a geometric with parameter $1/q$, that is, $$\P(W=w)=\frac{1}{q^{w-1}}\left(1-\frac{1}{q}\right),$$ for $w \geq 1$. By conditioning on $W$ and calculating the characteristic function, it follows that $\sum_{j=1}^W W_{j}$ has a exponential distribution with mean $r(q)/p_{N i}$, where $r(q)=q/(q-1)$. 

Now, let $\cP_{Ni}^{-L_{Ni}}(t):=\cP_{Ni}^{-L_{Ni}}([0, t]\times [0, 1])$, and define the counting process $X_N:=(X_N(t), t\geq 0)$ as $$X_N(t):=\sum_{i=1}^\infty |\cP_{Ni}^{-L_{Ni}}(t/s_N)|.$$ 
The above series is bounded by $|\cP(t/s_N)|$ and so it converges. Note that $(s_NT_{N1},  s_NT_{N2}, \ldots, s_NT_{Nm})$ are the arrival times of this process. As $\frac{T_{Nm}}{\tau_{Nm}}\pto 1$, for all $m \geq 1$, by the standard theory of weak convergence of point processes (Daley and Vere-Jones \cite{daleyverejones}, Theorem 9.1.VI) it is enough to show that the processes $X_N$ converge weakly to $\sM$.

The process $\cP_{Ni}(\cdot)$ is a homogeneous Poisson process of rate $p_{N i}$, with compensator $(p_{N i}r, r\geq 0)$. Thus, the process $(\cP_{Ni}(\cdot/s_N), t\geq 0)$ has compensator $(\psi_{Ni}t, t\geq 0)$ and the compensator of $\cP_{Ni}^{-L_{Ni}}(\cdot/s_N)$ is $C_{Ni}(t)=\psi_{Ni}(t-s_NF_{Ni}')_+$, where $s_NF_{Ni}'$ has a exponential distribution with mean $r(q)/\psi_{Ni}$. Consider the following three cases:
\begin{description}
\item[{\it Case} 1]$\lim_{N \rightarrow \infty}\psi_{N1}=0$. For $N, i\geq1$ let $\cF^{Ni} := (\cF^{Ni}_t, t\geq 0)$ be the natural filtration of $\cP_{Ni}(\cdot/s_N)$  and let $\cF^N$ be the smallest filtration containing $\{\cF^{Ni}: i\geq 1\}$.  Let $(C_{Ni}(t), t\geq 0)$ be the compensator of $\cP_{Ni}^{-L_{Ni}}(\cdot/s_N)$ with respect to the filtration $\cF^{Ni}$ and $(C_N(t), t\geq 0)$ the compensator of $X_N$ with
respect to $\cF^N$. Thus, $C_N(t) =\sum_iC_{Ni}(t)$. 
$$\E(C_N(t))=\sum_i\left(e^{-\psi_{Ni}\frac{t}{r(q)}}-1+\psi_{Ni}\frac{t}{r(q)}\right),$$
and
$$\Var(C_N(t))=\sum_i\left(1-e^{-2\psi_{Ni}\frac{t}{r(q)}}-2\psi_{Ni}\frac{t}{r(q)}e^{-\psi_{Ni}\frac{t}{r(q)}}\right).$$
Now, by elementary inequalities as in \cite[Lemma 11]{camarripitmanejp} it can be shown that  $\E(C_N(t))\rightarrow t^2/2r(q)^2$ and $\Var(C_N(t))\rightarrow 0$ for $t > 0$. This implies that $X_N$ converges weakly to the inhomogeneous Poisson process of rate $t/r(q)$ at time $t$, as required.

\item[{\it Case} 2]$\sum_{i}\psi_i^2< 1$. Let $(j_N \geq 1)$ be such that
$\lim_{N\rightarrow \infty}\sum_{i \leq j_N}\psi_{Ni}^2=\sum_i \psi_i^2$. Define the process $X_N^*(t):=\sum_{i> j_N} |\cP_{Ni}^{-L_{Ni}}(t/s_N)|$, and 
$X_{Ni}(t) = |\cP_{Ni}^{-L_{Ni}}(t/s_N)|$. Clearly, $X_{Ni}$ converges weakly to $B_i^{-L_i}$. Moreover, as in the previous case it can be shown that $X_N^*(s_Nt/s_N^*)$ converges weakly to the inhomogeneous Poisson process of rate $t/r(q)$ at time $t$. As $(s_N^*/s_N)^2\rightarrow 1-\sum_{i}\psi_i^2$, independence then implies that
$$(X_N^*, X_{N1}, \ldots, X_{Nj_N}, 0, 0, \ldots)\dto (B^*, B_1^{-L_1}, B_2^{-L_2}, \ldots).$$

\item[{\it Case} 3]$\sum_{i}\psi_i^2=1$. Let $(j_N \geq 1)$ be a sequence with $\lim_{N\rightarrow \infty}\sum_{i \leq j_N}\psi_{Ni}^2=1$. 
Define the process $X_N^*(t)$, and $X_{Ni}(t)$ as before. Clearly, $X_{Ni}$ converges weakly to $B_i^{-L_i}$. Now, it is easy to show that the compensator $C_N^*(t)$ of $X_N^*(s_Nt/s_N^*)$ satisfies: $\E(C_N^*(t))\rightarrow 0$ and $\Var(C_N^*(t))\rightarrow 0$, and the result follows.
\end{description}

\subsubsection{Corollary for the Uniform Case} Under the uniform distribution, that is, $p_{N i}=1/N$, for $i\in\{1, 2, \ldots, N\}$, the limiting distribution of the first collision time $T_{N1}$ is a Rayleigh distribution.

\begin{cor}Suppose there are $q\geq 2$ colors and $p_{N i}=1/N$, for $i \in[N]$, then the distribution of $T_{N1}/\sqrt{N}$ is the Rayleigh distribution with parameter $\sqrt{1-1/q}$, that is, 
$$\lim_{N\rightarrow \infty}\P(T_{N1}/\sqrt{N}> r)=e^{-\frac{1}{2}\left(\frac{q-1}{q}\right)r^2}.$$
Moreover, the distribution of $T_{Nm}/\sqrt{N}$ converges to $\sqrt{\frac{q}{q-1}\cdot \chi^2_{2m}}$, where $\chi^2_{2m}$ is a Chi-squared distribution with $2m$ degrees of freedom.
\label{cor:uniform}
\end{cor}

\begin{proof}The limiting distribution of $T_{N1}/\sqrt{N}$ follows directly from Theorem \ref{th:bipartitefirstrepeat}. Theorem \ref{th:bipartitejointrepeat} implies that the limiting distribution of $T_{Nm}/\sqrt{N}$ is the time of the $m$-th arrival of an inhomogeneous Poisson process $\cP_{\lambda}$ with rate $\lambda(t)=t/r(q)$, where $r(q)=q/(q-1)$. Denote the $m$-th arrival time by $\kappa_m$, and $|\cP_\lambda(t)|$ the number of arrivals in $\cP_{\lambda}$ up to time $t$. Using $|\cP_\lambda(t)| \sim \dPois(t^2/2r(q))$ and $\P(\kappa_m < t)=\P(|\cP_\lambda(t)| \geq m)$, the result follows.
\end{proof}


\subsection{Examples}
In this section connections of Theorem \ref{th:bipartitefirstrepeat} to the famous birthday problem are discussed. Other examples involving non-uniform urn selection probabilities are also given, illustrating the generality of the above results. 

\begin{example}(Birthday Problem) The classical birthday problem asks for the minimum number of people in a room such that two of them have the same birthday with probability at least 50\%. It is well known that the minimum number people for which this holds is approximately 23. In fact, the expected number of samples, chosen uniformly with replacement, required from a set of size $N$ until some value is repeated, is asymptotically $\sqrt{\pi N/2}$.  A generalization of this considers birthday coincidences among individuals of different types, that is, in a room with equal numbers of boys and girls, when can one expect a boy and girl to share the same birthday. Finding matches among different types can also be stated in terms of sampling colored balls and placing them in urns: Suppose there are $N$ urns and two colors, and the balls are colored independently with probability 1/2 and placed in the urns uniformly with probability $1/N$. The number of draws needed to have 2 balls with different colors in the same urn is the first time when a boy and a girl share the same birthday, when boys and girls sequentially enter a room independently with probability 1/2. In this case, Corollary \ref{cor:uniform} for $q=2$ shows that the limiting distribution of the first collision time is a Rayleigh distribution with parameter $\sqrt{1/2}$. This implies that the expected time of the first collision is $\sqrt{\pi N}$. When $N=365$ and the birthdays are assumed to be uniformly distributed over the year, the expected time before there is a boy and a girl with the same birthday is 34. 
For a detailed discussion on the birthday problem and its various generalizations and applications, refer to \cite{aldous,barbourholstjanson,dasguptasurvey,diaconisholmes,diaconismosteller} and the references therein.\end{example}


The following two examples exhibit the range of distributions that can be obtained from Theorem \ref{th:bipartitefirstrepeat} when the urn selection distribution is non-uniform.

%
%

%

\begin{example}\label{exnu1} Consider the probability distribution
\begin{equation}
p_{N1}=\frac{1}{\sqrt{N}}, \quad \text{ and } \quad p_{N i}=\frac{c_N}{N}, \text{ for } i \in [2,  N+1],
\label{eq:distributioN sqrt}
\end{equation}
where is $c_N=1-\frac{1}{\sqrt{N}}$, is such that $\sum_{i}p_{N i}=1$. Note that $c_N\rightarrow 1$, as $N \rightarrow \infty$, and $\psi_1=1/\sqrt 2$ and $\psi_i=0$ for all $i \in [2, N+1]$.
Therefore, by Theorem \ref{th:bipartitefirstrepeat}, 
$$\lim_{N\rightarrow \infty}\P\left(s_NT_{N1} > r\right) = e^{-\frac{1}{4}\left(\frac{q-1}{q}\right)r^2} e^{-\left(\frac{q-1}{q}\right) \frac{r}{\sqrt 2}}\left(q-(q-1)e^{-\frac{r}{q\sqrt 2}}\right).$$
Note that in this case $\sum_{i=1}^\infty \psi_i^2<1$, so in~(\ref{eq:qcolorsfirstrepeatth}) both the exponential term outside the product, and the terms inside the product are non-vanishing. 
For $q=2$ the limiting distribution has the following simpler form, 
\begin{equation*}
\lim_{N\rightarrow \infty}\P\left(s_NT_{N1} > r\right) =e^{-\frac{r^2}{8}}\left(2e^{-\frac{r}{2\sqrt 2}}-e^{-\frac{r}{\sqrt 2}}\right).
\end{equation*}
\end{example}

\begin{example}\label{exnu2}
Consider the following non-uniform urn selection distribution
\begin{equation}
p_{N1}=\frac{1}{\log{N}}, \quad \text{ and } \quad p_{N i}=\frac{c_N}{N}, \text{ for } i \in [2,  N+1],
\label{eq:distributioN log}
\end{equation}
where is $c_N=1-\frac{1}{\log N}$, is such that $\sum_{i}p_{N i}=1$. Note that $c_N\rightarrow 1$, as $N \rightarrow \infty$, and in this case $\psi_1=1$ and $\psi_i=0$ for all $i \in [2, N+1]$.
Therefore, 
$$\lim_{N\rightarrow \infty}\P\left(s_NT_{N1} > r\right) =  e^{-\frac{(q-1)}{q}r}\left(q-(q-1)e^{-\frac{r}{q}}\right).$$
Note that in this case $\sum_{i=1}^\infty \psi_i^2=1$, and so the exponential term in~(\ref{eq:qcolorsfirstrepeatth}) outside the product vanishes. For $q=2$ the limiting distribution simplifies to
$$\lim_{N\rightarrow \infty}\P\left(s_NT_{N1} > r\right) =2e^{-\frac{1}{2}\cdot r}-e^{-r}.$$
\end{example}

\section{Limiting Distributions of $m$-Fold Collision Times} 
\label{sec:mfold}

Recall the urn model in Definition~\ref{defn:uniform}, and analogous to the first collision time $T_{N1}$, define the $m$-{\it fold collision time} 
$T_{N1, m}$ as the first time there exists an urn with $m$ balls of color $a$  and $m$ balls of color $b$, for some $a\ne b \in [q]$. 

The next theorem gives the asymptotic distribution of $T_{N1, m}$. Calculations are similar to those in the proof of Theorem \ref{th:bipartitefirstrepeat}, and some details are omitted.

\begin{thm}Let $m\geq 1$ be a fixed postive integer, and $$s^{(2m)}_{N}=\left(\sum_i p_{N i}^{2m}\right)^{\frac{1}{2m}}, \text{ and } \psi_{Ni}^{(2m)}=\frac{p_{N i}}{s^{(2m)}_{N}}.$$ Suppose $\lim_{N\rightarrow \infty}p_{N1}=0$, and $\psi_i^{(2m)}:=\lim_{N\rightarrow \infty}\psi_{Ni}^{(2m)}$ exist for each $i \in \mathbb N$. Then for $r \geq 0$, \begin{align}
\lim_{N \rightarrow \infty}\P(&s^{(2m)}_{N}T_{N1, m}>r)\nonumber\\
=&e^{-\beta_m r^{2m}}\prod_i \left\{h_m\left(\frac{\psi_{i}^{(2m)}r}{q}\right)^{q-1}\left[q-(q-1)h_m\left(\frac{\psi_{i}^{(2m)}r}{q}\right)\right]\right\},
\label{eq:bipartitemlimitdistribution}
\end{align}
where $\beta_m=\left(1-\sum_{i=1}^\infty \left(\psi_i^{(2m)}\right)^{2m}\right)\frac{(q-1)}{2q^{2m-1}(m!)^q}$ and $h_m(x)=\sum_{y=0}^{m-1}e^{-x}\frac{x^y}{y!}$.
\label{th:bipartitemrepeat}
\end{thm}

\begin{proof}We consider the same embedding of the process as before: let $\cP$ be a homogeneous Poisson process on  $[0, \infty)\times [0, 1]$ of rate 1 per unit area, and  $\cP_{Ni}$ and $\cP_{Ni}^a$ be as defined in~(\ref{eq:colorstrip}). 
The $m$-fold collision time can be defined as in~(\ref{eq:TNm}) in terms of continuous-time process, and let $\tau_{N1, m}=\inf \{t: |\cP(t)|\geq  T_{N1, m}\}$.  By the strong law of large numbers, $|\cP(t)|/t$ converges almost surely to 1 as $t \rightarrow \infty$. As $\tau_{N1, m}\geq \tau_{N1}$, for $m \geq 1$, by Lemma \ref{lm:continuoustodiscrete}, $\lim_{N\rightarrow \infty}\tau_{N1, m}=\infty$ whenever $p_{N1}\rightarrow 0$. This implies $\frac{T_{N1, m}}{\tau_{N1, m}}\pto 1$, whenever $p_{N1}\rightarrow 0$, as $|\cP(\tau_{N1, m})|=T_{N1, m}$. 

Therefore, it suffices to derive the asymptotic distribution of $\tau_{N1, m}$. By definition, $\tau_{N1, m}$ is $$\inf \{t\geq 0: \exists j\in \mathbb N \text{ with } |\cP_{Nj}^\alpha(t)|\geq m\text{ and } |\cP_{Nj}^\beta(t)|\geq m, \text{ for } \alpha\ne \beta \in [q]\}.$$ 
This implies 
\begin{align}
\P(\tau_{N1, m}> r)=&\prod_{i}\left[q\cdot h_m\left(\frac{p_{N i}r}{q}\right)^{q-1}\left(1-h_m\left(\frac{p_{N i}r}{q}\right)\right)+h_m\left(\frac{p_{N i}r}{q}\right) ^{q}\right]\nonumber\\
=&\prod_i \left\{h_m\left(\frac{p_{N i}r}{q}\right)^{q-1}\left[q-(q-1)h_m\left(\frac{p_{N i}r}{q}\right)\right]\right\},
\label{eq:tN11}
\end{align}
where $h_m(x)=\sum_{y=0}^{m-1}e^{-x}\frac{x^y}{y!}$. Note that 
\begin{align}
(q-1)\log  h_m(x)&+\log(q-(q-1)h_m(x))\nonumber\\
=&(q-1)\log (1-(1-h_m(x)))+\log(1+(q-1)(1-h_m(x)))\nonumber\\
=&-\sum_{k=2}^\infty\frac{(q-1)+(-1)^k(q-1)^k}{k}\cdot\left(1- h_m(x)\right)^{k}\nonumber\\
=&-\sum_{k=2}^\infty\frac{(q-1)+(-1)^k(q-1)^k}{k}\cdot\left(\sum_{y=m}^{\infty}e^{-x}\frac{(x)^y}{y!}\right)^{k}\nonumber\\
=&-\frac{q(q-1)}{2}\cdot\frac{x^{2m}}{(m!)^2}+\sum_{k=2m+1}^\infty (-1)^{k+1}a_kx^k.
\label{eq:tN2}
\end{align}
The interchange of the different summations is justified by the absolute convergence of the series, which can be proved by  arguments similar to those in Lemma \ref{lm:series}. Combining~(\ref{eq:tN11}) and~(\ref{eq:tN2}) we get
\begin{align}
\log \P(s^{(2m)}_{N}& \tau_{N1, m}> r)\nonumber\\
=&-\frac{q(q-1)}{2}\cdot\frac{r^{2m}}{q^{2m}(m!)^q}+\sum_{k=2m+1}^\infty (-1)^{k+1}a_kr^kq^{-k}\sum_i\left(\psi_{Ni}^{(2m)}\right)^k,\nonumber
\end{align}
since $\sum_i\left(\psi_{Ni}^{(2m)}\right)^{2m}=1$. Finally, $\left(\psi^{(2m)}_{Ni}\right)^k\rightarrow \left(\psi_i^{(2m)}\right)^k$ for $k > 2m$, and using absolute convergence,  $\lim_{N\rightarrow \infty}\log \P(s^{(2m)}_{N}\tau_{N1, m}> t)$ simplifies to
$$-\beta_m t^{2m}+\sum_i \left\{h_m\left(\frac{\psi_{i}^{(2m)}t}{q}\right)^{q-1}\left[q-(q-1)h_m\left(\frac{\psi_{i}^{(2m)}t}{q}\right)\right]\right\},\nonumber$$ 
where $\beta_m$ is as defined~\ref{eq:bipartitemlimitdistribution}. This completes the proof of the result.
\end{proof}

\section{Generalizing the Urn Model: Proof of Theorem \ref{th:qcolorurn}}
\label{sec:qcolor}

The proof of Theorem \ref{th:qcolorurn} presented below is similar to that of Theorem \ref{th:bipartitefirstrepeat} but requires more careful calculations. To this end, recall the urn model from~Definition \ref{defn:nonuniform} with non-uniform color and non-uniform urn selection probabilities. As before, the collision time $T_{N1}$ is  the first time that there exist two balls with different colors in the same urn.

\subsection{Proof of Theorem \ref{th:qcolorurn}}

Let $\P_N$ be a ranked discrete distribution and $\pmb c=(c_1, c_2, \ldots, c_q)$ be the coloring distribution as in Definition~\ref{defn:nonuniform}. Let $\cP$ be a homogeneous Poisson process on $\sS:=[0, \infty )\times[0, 1]$ of rate 1 per unit area, with points $\{(S_{1}, W_{1}), (S_{2}, W_{2}), \ldots \}$, where $0 < S_{1} < S_{2} <\ldots $ are the points of a homogeneous Poisson process on $[0, \infty)$ of rate 1 per unit length, and $W_1, W_2, \ldots $ are i.i.d. $\dU(0, 1)$. Let $\cR_t=[0, t]\times[0, 1]$ and $\cP(t)$ be the restriction of $\cP$ to $\cR_t$.

\begin{itemize}

\item Color the points in $\cP$ independently with one of $q$ colors, $\{1, 2, \ldots, q\}$ as follows: $$\P((S_i, W_i)\in \cP \text{ has color } a\in [q])= c_a,$$ independently over the points in $\cP$. For $a\in [q]$ denote by $\cP^a$ the subsets of $\cP$ colored $a\in [q]$. By the marking theorem \cite{poisson}, $\cP^1, \cP^2, \ldots, \cP^q$ are independent Poisson process with rates $c_a$ on $\cR$, respectively.
 
\item  For each $a\in [q]$ and $N\geq 1$, partition $[0, 1]$ into intervals $J_{N 1, a},$ $J_{N 2, a}, \ldots, $ such that the length of $J_{N i, a}$ is $p_{N i, a}$. For $t\geq 0$, let 
\begin{equation}
\cP^a_{N i}=\cP^a\cap [0, \infty)\times J_{N i, a}.
\label{eq:colorstripgen}
\end{equation}
Clearly, $\cP_{N 1}^a, \cP_{N 2}^a, \ldots$ are independent Poisson processes with rates $c_a p_{N 1}, c_a p_{N 2}, \ldots$, respectively.
\end{itemize}

Recall the definition of the first collision time $T_{N1}$ for the urn model in Definition~\ref{defn:nonuniform}. It can be described in terms of the above process as follows: let $C_{j}$ be the color of the point $(S_j, W_j)$ and $Z_{Nj}=\sum_{i} i \pmb 1\{ W_j\in J_{Ni, C_j}\}$. The sequence $\{(C_j, Z_{Nj})\}_{j \geq 1}$ in the discrete time model corresponds to the color of the $j$-th ball and the urn to which the $j$-th ball is assigned. In particular, the first collision time $T_{N1}$ is 
$$\inf\{j \in \N: \exists n \in \N \text{ with } Z_{Nj}=Z_{Nj'}=n \text{ and } C_{j'}\ne C_{j}, \text{ for some } j'<j \}.$$


\begin{lem}Let $\tau_{N1}=\inf \{t: |\cP(t)|\geq T_{N1}\}$. Then $\frac{T_{N1}}{\tau_{N1}}\pto 1$, whenever $\lim_{N\rightarrow \infty}\max_ip_{Ni, a}=0$, for all $a\in [q]$.
\label{lm:continuoustodiscreteg}
\end{lem}

\begin{proof}
By the strong law of large numbers $|\cP(t)|/t$ converges almost surely to 1 as $t \rightarrow \infty$. Therefore, it suffices to show $\tau_{N 1}$ converges in probability to infinity as $N \rightarrow \infty$, since $|\cP(\tau_{N 1})|=T_{N 1}$. 

By definition $\tau_{N1}$ is $$\inf \{t\geq 0: \exists~ i\in \mathbb N \text{ with } |\cP_{Ni}^\alpha(t)|>0, ~|\cP_{Ni}^\beta(t)|>0, \text{ for some } \alpha\ne \beta \in [q] \},$$
where $ \cP_{N j}^a(t)$ is the restriction of $\cP_{N j}^a$ (defined in~(\ref{eq:colorstripgen})) to $\cR_t$, for $a\in [q]$.  This implies that 

\begin{eqnarray}
\P(\tau_{N1}> t)&=&\prod_i \left(\sum_{a=1}^q\left(1-e^{-p_{Ni, a}c_a t}\right)\prod_{b\ne a}e^{-p_{Ni, b}c_b t}+\prod_{a=1}^qe^{-p_{Ni, a}c_a t}\right)\nonumber\\
&=&\prod_i \left(\sum_{a=1}^q e^{-c_a t\sum_{b\ne a}p_{Ni, b}}-(q-1)e^{- t\sum_{a=1}^q p_{Ni, a}c_a}\right)\nonumber\\
&=&e^{-t \sum_i \sum_{a=1}^qc_a p_{Ni, a}} \prod_i\left(1+\sum_{a=1}^q e^{p_{Ni, a}c_a t}-q\right).
\label{eq:lmcontinuousgeneral}
\label{eq:tNgen}
\end{eqnarray}
As $\max_ip_{Ni, a}\rightarrow 0$, choose $N> N(t,a)$ so that $p_{Ni, a}< \frac{1}{c_at}\log\left(1+\frac{1}{q}\right)$ for all $i \in \N$. Therefore, for $N \geq \max_a N(t,a)$ and some constant $C>0$
\begin{align}
\log & \P(\tau_{N1}> t)\nonumber\\\
\geq&\sum_i \left\{-t \sum_{a=1}^qp_{Ni, a}c_a+\left(\sum_{a=1}^q e^{p_{Ni, a}c_at}-q\right)-\frac{1}{2}  \left(\sum_{a=1}^q e^{p_{Ni, a}c_at}-q\right)^2\right\}\nonumber\\
\geq&\frac{t^2}{2}  \sum_i \sum_{a}c_a^2 p_{Ni, a}^2-\frac{C t^2}{2}  \sum_i \left(\sum_{a}c_a p_{Ni, a}\right)^2.\nonumber
\end{align}
The first inequality uses $\log(1+x)\geq x-\frac{x^2}{2}$, for $|x|<1$, and the second uses: (a) $e^x\geq 1+x+\frac{x^2}{2}$ on the first exponential term and (b) $e^x\leq 1+Cx$, for some constant $C:=C(q)$ when $|x|\leq \log(1+1/q)$ on the second exponential term. 

Now, as $\sum_i p_{Ni, b}=1$, for all $b \in [q]$, $$\sum_i \left(\sum_{a} p_{Ni, a}\right)^2\leq  (\max_{i\in \N} \sum_{a}p_{Ni, a}) \sum_{a} \sum_i p_{Ni, a}=q\max_{i\in \N} \sum_{a}p_{Ni, a}.$$ Therefore,
\begin{eqnarray}
|\log \P(\tau_{N1}> t)|\leq \frac{C t^2}{2}  \sum_i \left(\sum_{a}c_a p_{Ni, a}\right)^2&\leq & \frac{C t^2}{2} \max_{a\in [q]}c_a^2  \sum_i \left(\sum_{a} p_{Ni, a}\right)^2\nonumber\\
&\leq & \frac{C q t^2}{2} \max_{a\in [q]}c_a^2 \max_{i\in \N}\sum_{a} p_{Ni, a}\nonumber\\
&\rightarrow & 0,\nonumber
\end{eqnarray}
and the result follows. 
\end{proof}

With $s_N=\left(\sum_i (\sum_{a=1}^q c_ap_{Ni, a})^2\right)^\frac{1}{2}$, $\psi_{Ni, a}$ as defined in the statement of the theorem, and~(\ref{eq:tNgen}) 
\begin{align}
\log\P(s_NT_{N1}> r)=&\sum_{i}\log\left\{e^{-r \sum_i \sum_{a=1}^qc_a\psi_{Ni, a}} \left(1+\sum_{a=1}^q e^{r c_a\psi_{Ni, a}}-q\right)\right\}\nonumber\\
:=&g(r, \vec \psi_{N, 1}, \ldots, \vec \psi_{N, q}),
\end{align}
where $\vec \psi_{N, a}=(\psi_{N1, a}, \psi_{N2, a}, \ldots )$, for $a\in [q]$ and the function $g$ is defined in~(\ref{eq:g}). 

As $\sum_{i}\psi_{N i, a}^2 <\infty$ and $\lim_{N\rightarrow \infty}\psi_{N i, a}=\psi_{i, a}$ exists for all $i$ and $a\in [q]$, by Fatou's lemma $\sum_{i}\psi_{i, a}^2<\infty$. Therefore, for $a\in [q]$, $\lim_{i}\psi_{i, a}=0$, and for $r>0$ there exists $N(r), j(r)$ such that $$\psi_{N j(r), a}<\frac{1}{r}\log(1+1/q) \text{ for all } N > N(r).$$ 

Let $A$ and $B$ be the functions defined in Lemma~\ref{lm:seriesq}. Define $$\Gamma=\bigcup_{k=1}^\infty\{(\gamma_1, \gamma_2, \ldots, \gamma_k)\in \mathbb N^k: \sum_{b=1}^k\gamma_b\geq 3\}.$$
For $(\gamma_1, \gamma_2, \ldots, \gamma_k)\in \Gamma$ and $i> j(r)$,
\begin{align}
\prod_{b=1}^k \sum_{a=1}^q c_a^{\gamma_b}\psi_{Ni, a}^{\gamma_b}\leq & \left(\sum_{a=1}^q c_a\psi_{Ni, a}\right)^{\sum_{b=1}^a\gamma_b}\nonumber\\
\leq & \left(\max_{a\in[q]}\max_{i>j(r)}c_a \psi_{N i, a}\right)^{\sum_{b=1}^a\gamma_b-2}\left(\sum_{a=1}^q c_a\psi_{Ni, a}\right)^2.\nonumber
\end{align}
Using this and Lemma \ref{lm:seriesq} it follows that
\begin{eqnarray}
\lim_{N\rightarrow \infty}A(r, \vec \psi_{N, a}, \vec \psi_{N, a}, \ldots, \vec \psi_{N, q})&=&A(r, \vec \psi_{1}, \vec \psi_{2}, \ldots, \vec \psi_{q}),\nonumber\\
\lim_{N\rightarrow \infty}B(r, \vec \psi_{N, a}, \vec \psi_{N, a}, \ldots, \vec \psi_{N, q})&=&B(r, \vec \psi_{1}, \vec \psi_{2}, \ldots, \vec \psi_{q}).
\label{eq:limit1}
\end{eqnarray}

Finally,  by assumption $\lim_{N\rightarrow \infty}\sum_i \psi_{Ni, a}^2=\phi_a$ exists for all $a\in [q]$, and hence, as $N\rightarrow \infty$, \begin{eqnarray}
\left(\sum_i\left(\sum_{a=1}^q c_a\psi_{Ni, a}\right)^2-\sum_{a=1}^q c_a^2\sum_i\psi_{Ni, a}^2\right)&\rightarrow&1-\sum_{a=1}^q c_a^2\left(\lim_{N \rightarrow \infty}\sum_i \psi_{Ni, a}^2\right)\nonumber\\
&=&1-\sum_{a=1}^q c_a^2\phi_a.
\label{eq:limit2}
\end{eqnarray}
Combining Equations (\ref{eq:limit1}) and (\ref{eq:limit2}) and using Lemma \ref{lm:seriesq} the result follows.


\subsection{Useful Corollaries and Examples Continued}

A special case of Theorem~\ref{th:qcolorurn} is to consider the case where  $p_{Ni, a}=p_{N i}$, for all $a\in [q]$, and 
a general coloring distribution  $\pmb c=(c_1, c_2, \ldots, c_q)$. This simplifies~(\ref{eq:qcolorsgen}) to the following:

\begin{cor}For $a\in [q]$, and $N, i \in \mathbb N$, let $s^2_{N}=\sum_{i}p_{N i}^2, \text{ and } \psi_{ni}=\frac{p_{N i}}{s_{N}}$. Suppose that $\lim_{N\rightarrow \infty}p_{N1}=0$ and $\psi_{i}=\lim_{n\rightarrow \infty}\psi_{N i}$ exists, for each $i \in \mathbb N$. Then, as $N\rightarrow \infty$,
\begin{equation*}
\P(s_{N}T_{N1}> r) \rightarrow e^{-\frac{1}{2}\left(1-\sum_{i}\psi_i^2\right)(1-\sum_{a=1}^q c_a^2)r^2}\prod_ie^{-r\psi_i} \left(1+\sum_{a=1}^q e^{\psi_{i}c_a r}-q\right).
\end{equation*}\vspace{-0.15in}
\label{cor:qcolorurn}
\end{cor}

The main application of Theorem \ref{th:qcolorurn} is in deriving the limiting distributions of the running times of algorithms for the discrete logarithm problem (DLP) in an interval. To this end, assume $q=2$ colors and consider 2 discrete ranked distributions $p_{N1}\geq p_{N2}\geq \ldots$ and $q_{N1}\geq q_{N2}\geq \ldots$; where at each  step one of the two colors is chosen with probability 1/2. If color 1 is chosen then a ball of color 1 is put in the $i$-th urn with probability $p_{N i}$, otherwise a ball of color 2 is put in the $i$-th urn with probability $q_{Ni}$. 

\begin{cor}\label{cor:2colorurn}
For $a\in [q]$, and $N, i \in \mathbb N$, let 
\begin{equation}
s_{N}=\frac{1}{2}\left(\sum_{i}\left(p_{N i}+q_{Ni}\right)^2\right)^{\frac{1}{2}}, \quad \psi_{Ni}=\frac{p_{N i}}{s_{N}}, \text{ and } ~\theta_{Ni}=\frac{q_{Ni}}{s_{N}}.
\label{eq:sn}
\end{equation}
Suppose $p_{N1}\rightarrow 0$, $q_{N1}\rightarrow 0$, as $N\rightarrow \infty$, and $\lim_{N\rightarrow \infty}\psi_{Ni}=\lim_{N\rightarrow \infty}\theta_{Ni}=0$, for all $i \in \mathbb N$, and $\phi_1=\lim_{N\rightarrow \infty}\sum_i\psi_{Ni}^2$, $\phi_2=\lim_{N\rightarrow \infty}\sum_i\theta_{Ni}^2$ exists. Then 
\begin{equation} \label{eq:qcolorsfirstrepeatcor}
\lim_{N\rightarrow \infty}\P(s_{N}T_{N1}> r)=e^{-\left(1-\frac{1}{4}\sum_{a=1}^2\phi_a\right)\frac{r^2}{2}}.
\end{equation}
\end{cor}
%

The setup of Theorem \ref{th:qcolorurn} is very general and it can be used in various applications. However, when the urn selection distribution depends on the color of the ball, sometimes the scaling in Theorem \ref{th:qcolorurn} may not give a nontrivial limiting distribution as indicated in the following example.
 
\begin{example}
Let $p_{N1}\geq p_{N2}\geq \ldots $ be the probability distribution (\ref{eq:distributioN log}), and consider the following process: Every time choose one of two colors independently with probability 1/2; if color 1 is chosen, then with probability $p_{N i}$ a ball colored 1 goes to the $i$-th urn, otherwise color 2 is chosen and a ball with that color goes to the $i$-th urn with probability $\frac{1}{N+1}$. Let $T_{N1}$ be the first collision time. In this case, $s_N\log N\rightarrow 1$, where $s_N$ is as defined in (\ref{eq:sn}), and Theorem \ref{th:qcolorurn} gives $s_NT_N=T_N/\log N$ converges to infinity in probability. However, in this case, it can be easily shown that 
$$\lim_{N\rightarrow \infty}\P(T_N/\sqrt N\geq r)=e^{-\frac{r^2}{4}},$$
the Rayleigh distribution with parameter $\sqrt 2$. 
\end{example}

\section{Algorithms for the Discrete Logarithm Problem: Limiting Distribution of Running Times}
\label{sec:dlp}

The central idea of the Gaudry-Schost (GS) algorithm, as well as, the the kangaroo algorithm of Pollard is based on the collision time of 2 independent pseudo-random walks. Let $g$ and $h$ be the DLP instance, with $h = g^a$ for some integer $-N/2 \leq  a \leq N/2$, where $N$ is the size of the interval. The cyclic group $G$ generated by $g$ will often be described in terms of the exponent space. Define the {\it tame set} $T = [-N/2, N/2]$ and the {\it wild set} $W = a + T = \{a + b : b \in [-N/2, N/2]\}$.  A {\it tame walk} is a sequence of points $\{g^{a_i}\}_{i\geq 1}$ where $a_i\in T$ and a {\it wild walk} is a sequence of points $g^{b_i} = hg^{a_i}$ with $b_i\in W$.  Each walk proceeds until a distinguished point is hit. This distinguished point is then stored on a server, together with the corresponding exponent and a flag indicating which sort of walk it was. When the same distinguished point is visited by two different types of walk, there is a tame-wild collision giving an equation of the form $g^{a_i} = hg^{b_j}$, and the DLP is solved as $h = g^{a_i-b_j}$. 

The actual GS algorithm is much more complicated and although the starting
point of the pseudorandom walk will be random, inherently the rest of the steps are not random, and only a heuristic running time can be derived. Experimental evidence show that the pseudorandom walks get close enough to a random selection. Therefore, it is standard in the literature to assume that when $N$ is sufficiently large the pseudorandom walks performed by the algorithm is sufficiently random, and the running time can be analyzed by an idealized birthday problem involving the tame-wild collision. Throughout the paper we work with this assumption, and refer to the running times of these algorithms as {\it idealized} running times. 
Then, identifying the tame walks as being color 1 and the wild walks as being color 2, and the group elements as urns, the idealized running time of the GS algorithm is precisely when two balls of different colors are placed in the same urn. 

Generally, only the expectation of the collision time is used to quantify the performance of these algorithms, using the  birthday problem.
In the following theorem, the limiting distribution of the idealized running time of the GS algorithm for any problem instance is determined. It is assumed that the elements from $T$ and $W$ are sampled with probability 1/2 each, which means that at each step the two colors are chosen with probability 1/2 each. This is quite a realistic assumption as in practice one often considers distributed or parallel implementations of the algorithm \cite{birthday_discrete_logarithm,galbraithpkc}.

\begin{thm}Given an instance $(g, h)$ of the DLP with $h=g^{xN}$, where $x \in [-1/2, 1/2]$, the limiting distribution of the idealized running time $T_N^{(x)}$ of the GS algorithm is $$\lim_{N \rightarrow \infty}\P(T_N^{(x)}> r\sqrt{N})=e^{-\left(\frac{1-|x|}{2}\right)\frac{r^2}{2}}.$$\vspace{-0.15in}
\label{th:gaudryschost}
\end{thm}

\begin{proof}
By symmetry, it suffices to consider $0\leq x <1/2$. This implies that $|T\cap W|=(1-x)N$. Define $p_{N i}=1/N$, for $i \in T$, and $q_{N i}=1/N$, for $i \in W$. Then by~(\ref{eq:sn})
$$s_{N}=\sqrt{\frac{1-x/2}{N}}, \quad \text{ and } \quad \lim_{N\rightarrow \infty}\psi_{Ni}=\lim_{N\rightarrow \infty}\theta_{N i}=0.$$ Moreover, $\phi_1=\phi_2=\lim_{N}\sum_i \psi^2_{Ni}=\lim_{N\rightarrow \infty} \sum_i\theta^2_{N i}=\frac{1}{1-x/2}$.
Therefore, applying  Theorem \ref{th:qcolorurn} the result follows. 
\end{proof}

\begin{remark}\label{remgs} Theorem \ref{th:gaudryschost} shows $T_N^{(x)}/\sqrt{N}$ converges to a Rayleigh distribution with parameter $\left(\frac{2}{1-x}\right)^\frac{1}{2}$. Therefore, it is expected that in the limit 
$\E(T_N^{(x)})/\sqrt{N}\rightarrow (1-x)^{-\frac{1}{2}}\sqrt{\pi}$. (This can be made rigorous by showing the uniform integrability of the sequence $T_N^{(x)}/\sqrt N$, for example, by arguing that the second moment of $T_N^{(x)}/\sqrt N$ is bounded.) Assuming $x$ is uniformly distributed over $[-1/2, 1/2]$ gives, 
$2\int_{0}^{\frac{1}{2}}(1-x)^{-\frac{1}{2}}\sqrt{\pi N}=(4-2\sqrt 2)\sqrt{\pi N}\approx 2.08\sqrt{N}.$ 
This is the leading term of the expected heuristic running time of the GS algorithm averaged over all problem instances, which was proved earlier in \cite[Theorem 2]{galbraithpkc} using the birthday paradox.  
\label{rem:gsalgorithm}
\end{remark}

\subsection{Accelerated Gaudry-Schost (AGS) Algorithm} In groups where computing $h^{-1}$ for any group element $h$ is much faster than a general group operation the GS algorithm can be greatly  accelerated by performing random walks in sets of equivalence classes corresponding to the tame and wild sets. As before, let $N, g$ and $h$ be given such that
$4 | N$, $h = g^a$ and $-N/2 \leq a \leq N/2$. Define the tame and wild sets (as sets of equivalence classes) by
$$\tilde T = \{\{b,-b\} : b\in [-N/2, N/2]\}, ~ \tilde W = \{\{a + b, -(a + b)\} : b \in [-N/4, N/4]\}.$$
Note that $ |\tilde T | = 1+N/2 \approx N/2$. The algorithm samples alternately from $\tilde T$ and $\tilde W$ with probability $1/2$. 

\begin{thm}Given an instance $(g, h)$ of the DLP with $h=g^{xN}$, where $x \in [-1/2, 1/2]$, the limiting distribution of the idealized running time $T_N^{(x)}$ of the AGS algorithm of Galbraith and Ruprai \cite{galbraithpkc} is 
\begin{equation}
\lim_{N\rightarrow\infty}\P(T_N^{(x)}> r\sqrt{N})=\left\{
\begin{array}{cc}
e^{-\frac{r^2}{2}} &   \hbox{ if } |x|<1/4,\\
e^{-\left(3-4|x|\right)\frac{r^2}{4}} &    \hbox{ if } 1/4\leq |x|\leq 1/2.
\end{array}
\right.
\label{eq:dlpags}
\end{equation}
\label{th:dlpdistribution}\vspace{-0.1in}
\end{thm}

\begin{proof}Let  $(g, h)$ be the DLP instance with $h=g^{xN}$, and $T_N^{(x)}$ the idealized running time of the AGS algorithm of Galbraith and Ruprai \cite{galbraithpkc}. The analysis has two cases:

\begin{description}
\item[{\it Case} 1]$0 \leq x < 1/4$. In this case $\tilde W \subseteq \tilde T$ and the algorithm samples from $\tilde T$ and $\tilde W$ alternately and uniformly. This is equivalent to sampling uniformly from $[0, N/2]$ with probability 1/2, and sampling an element $b$ from $ [0, N/2]$ with probability $4/N$, for $0 \leq b < N/4-|x|N$, and probability $2/N$,  for $N/4-|x|N \leq b \leq  |x|N + N/4$, with probability 1/2. In this case, 
$$s_{N}=\sqrt{\frac{10-8x}{4N}}, \text{ and } \lim_{N\rightarrow \infty}\psi_{Ni}=\lim_{N\rightarrow \infty}\theta_{Ni}=0.$$ Moreover, $\phi_1= \frac{8}{10-8x}$ and $\phi_2=\frac{4(4-8x)}{10-8x}$.
Therefore,  applying  Theorem \ref{th:qcolorurn} it follows 
$$\lim_{N\rightarrow\infty}\P\left(T_N> r\sqrt{\frac{4N}{10-8x}}\right)=e^{-\left(\frac{1}{5-4x}\right)r^2}.$$

\item[{\it Case} 2]$1/4 \leq x \leq 1/2$. In this case, $|\tilde T\cap \tilde W|=N(3/4-|x|)$ (here $|\tilde T\cap \tilde W|$ refers to the number of equivalence classes in the intersection). 
The algorithm samples uniformly between the two sets $\tilde T$ and $\tilde W$, where $|\tilde T|=|\tilde W|\approx N/2$, and as in the proof of Theorem \ref{th:gaudryschost} the limiting distribution of the idealized running time can be obtained.
\end{description}
\end{proof}

\begin{figure*}[h]
\centering
\begin{minipage}[c]{0.49\textwidth}
\centering
\includegraphics[width=2.4in]
    {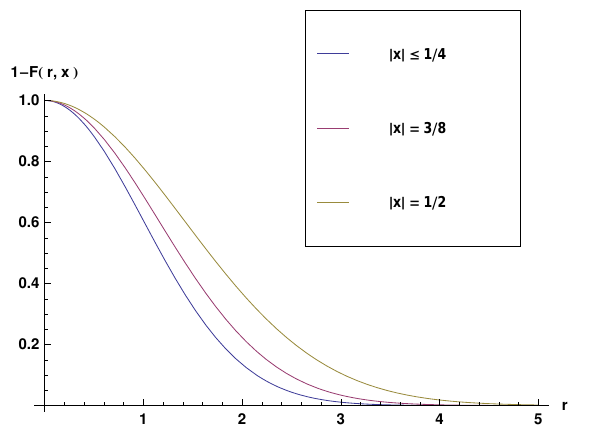}\\
\small{(a)}
\end{minipage}
\begin{minipage}[c]{0.49\textwidth}
\centering
\includegraphics[width=2.4in]
    {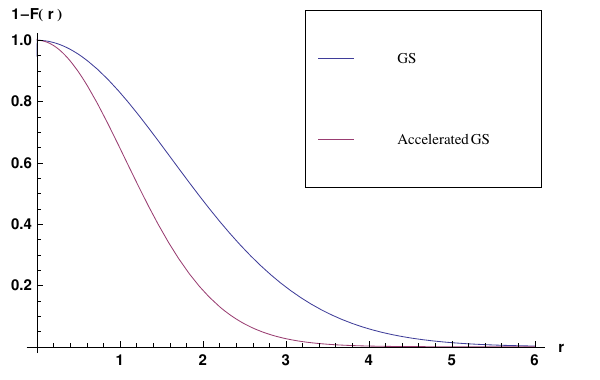}\\
\small{(b)} 
\end{minipage}
\caption{\small{(a) Limiting idealized hazard rate of the AGS algorithm for various problem instances, (b) Comparing limiting idealized hazard rates of GS and the AGS algorithms.}}\vspace{-0.1in}
\label{fig:plotalgorithms}
\end{figure*}

\begin{remark} Note that the limit~(\ref{eq:dlpags}) is a distribution function for every $x\in [-1/2, 1/2]$, where $x$ is the unknown exponent in the DLP problem. Generally $x$ is assumed to be uniformly distributed over $ [-1/2, 1/2]$ and the expected running time (as in Remark~\ref{remgs}, to make this rigorous, one has to argue that the sequence $T_N^{(x)}/\sqrt N$ is  uniformly integrable) will be
\begin{align*}
2\sqrt N&\left(\int_0^{1/4}\int_{0}^\infty e^{-\frac{r^2}{2}}drdx + \int_{1/4}^{1/2}\int_0^\infty e^{-\left(3-4|x|\right)\frac{r^2}{4}}dr dx\right)\nonumber\\
=& (5\sqrt{2}/4-1)\sqrt{\pi N}\approx 1.36\sqrt{N}.
\end{align*}
This is the leading term of the expected heuristic running time of the AGS algorithm averaged over all problem instances, which was proved by Galbraith and Ruprai (Theorem 4,  \cite{galbraithpkc}). Assuming the walks are truly random, Theorem \ref{th:dlpdistribution} gives the idealized asymptotic hazard rate of the AGS algorithm as $1-F(r, x)=\lim_{n \rightarrow \infty}\P(T_N^{(x)}/\sqrt{N}>r)$, quantifying which problem instances are easier/harder. Figure \ref{fig:plotalgorithms}(a) shows the asymptotic hazard rate for the AGS algorithm for various values $x\in [-1/2, 1/2]$. It is observed that as $x$ approaches $1/2$, the idealized running time of the AGS algorithm increases, which is expected as the intersection between the tame and wild sets decrease. Moreover, assuming that $x$ uniformly distributed over $[-1/2, 1/2]$, the performance of the different variants of the GS algorithms can be compared using the limiting idealized hazard rates averaged over all problem instances $$1-F(r)=\int_{-1/2}^{1/2}(1-F(r, x))dx.$$  Figure \ref{fig:plotalgorithms}(b) shows the limiting idealized hazard rates of the GS and AGS algorithms averaged over all problem instances. It shows that the limiting idealized running time of the GS algorithm stochastically dominates the AGS algorithm, that is, it is better than the GS algorithm not only in expectation but also for all values $r\geq 0$.
\label{rem:agsalgorithhm}
\end{remark}

\section{Collision Times in Sequential Graph Coloring Using Stein's Method}
\label{sec:steinlimit}

In this section Stein's method for Poisson approximation is used to determine the limiting distributions of the collision times for the preferential attachment model and the infinite path. To this end, recall the following version of Stein's method based on dependency graph:

\begin{thm}[Chatterjee et al. \cite{chatterjeediaconissteinsmethod}]
Suppose $\{X_i\}_{i\in \sI}$ is a finite collection of binary random variables with dependency graph $(\sI, \sE)$, that is,  $(X_i, X_j)\in \sE$, whenever $X_i, X_j$ are dependent.  Let $W=\sum_{i\in \sI} X_i$, $p_i=\P(X_i=1)$, $p_{ij}=\P(X_i=X_j=1)$, and $\lambda=\sum_{i\in \sI} p_i$.  Then\footnote{Note that $||W-\dPois(\lambda)||=\frac{1}{2}\sum_{i}|\P(W=i)-\frac{e^{-\lambda}\lambda^i}{i!}|$ is the total variation distance between $W$ and the Poisson distribution with mean $\lambda$.} $$||W-\dPois(\lambda)||\leq \min\big\{1, \frac{1}{\lambda}\big \}\left(\sum_{\substack{i \in \sI\\ j \in N(i)\backslash\{i\}} } p_{ij}+\sum_{\substack{i \in \sI \\ j \in N(i)}} p_ip_j\right).$$ 
\label{th:steinsdependencygraph}
\end{thm}

\subsection{Preferential Attachment Models: Proof of Theorem \ref{th:palimitdistribution}}
Recall the definition of the $PA(m)$ model and the graph sequence $(G_m^t)_{t\geq 1}$ from Section \ref{sec:paintro}. Denote by $S(G_{m}^t)$ the  underlying simple graph associated with $G_m^t$. Let $\sG=(G_m^{t})_{t\geq 1}$ and consider the coloring scheme described in Section~\ref{sec:paintro}. Let $T_{N1}^{PA(m)}$ be the first time there is a monochromatic edge in a sequential coloring of $\sG$. 

For $t\geq 1$ (possibly depending on $N$), let $X_{1, t}, X_{2, t}, \ldots $ be i.i.d. $\P_N$ ($X_{i, t}$ corresponds to the color of the vertex $i$ in $S(G_m^t)$). For $(i, j)\in E(S(G_m^t))$ define
$$Z_{(i, j)}^{(t)}=\pmb 1\{X_{i, t}=X_{j, t}\},$$ and $W_{t}=\sum_{(i, j)\in E(S(G_m^t))}Z_{(i, j)}^{(t)}$, which is the number of monochromatic edges in $S(G_m^t)$. Let $\lambda_{t}:=\E(W_{t})=|E(S(G_m^t))|\sum_{i}p_{N i}^2=(mt-o(t))\sum_{i}p_{N i}^2$. 

For two distinct edges $e_1=(i_1, j_1)$ and $e_2=(i_2, j_2)$,
$$\P(Z_{e_1}^{(t)}=Z_{e_2}^{(t)}=1)=
\left\{
\begin{array}{cc}
\left(\sum_{i}p_{N i}^2\right)^2 &   \text{ if } e_1\cap e_2 =\emptyset \\
\sum_{i}p_{N i}^{3} &   \text{ otherwise.}   \\
\end{array}
\right.
$$
since the random variables $Z_{e_1}^{(t)}$ and $Z_{e_2}^{(t)}$ are independent whenever the edges $e_1\cap e_2 =\emptyset $, the dependency graph associated with the random variables $\{Z_{(i, j)}^{(t)}\}_{(i, j)\in E(S(G_m^t)}$ can be constructed by putting an edge between $Z_{e_1}^{(t)}$ and $Z_{e_2}^{(t)}$ whenever  $e_1$ and $e_2$ share a vertex. Then the error term in (\ref{th:steinsdependencygraph}) becomes
\begin{eqnarray}\label{eq:steinpa}
\left|\P(W_{t}=k|G_{t})-e^{-\lambda_{t}}\frac{\lambda_{t}^k}{k!}\right|&\leq & C_{t} 2|T(S(G_m^t))| \left(\sum_{i}p_{N i}^{3}+\left(\sum_{i}p_{N i}^2\right)^2\right)\nonumber\\
& &\;\;\;\;\;\;\;\;\;\;\;+2C_t|E(S(G_m^t))|\left(\sum_{i}p_{N i}^2\right)^2,
\end{eqnarray}
where $C_{t}=\min(1, \lambda_{t}^{-1})$ and $T(S(G_m^t))$ is the number of 2-stars (the bipartite graph $K_{1, 2}$) in $S(G_m^t)$. 

If $t=t(N)$ is such that $\lim_{N\rightarrow\infty}t\sum_{i}p_{N i}^2=\lambda>0$ for some $\lambda>0$, then $\lambda_{t}\rightarrow m\lambda$, and $|E(S(G_m^t))|\left(\sum_{i}p_{N i}^2\right)^2 \leq p_{N1} \lambda_t \rightarrow 0$. Moreover, Bollob\'as \cite[Theorem 16]{bollobaspa} shows that 
$$(1-\varepsilon){m+1\choose 2}t\log t \leq |S(G_m^t)|\leq (1+\varepsilon){m+1\choose 2}t\log t,$$
with high probability as $t\rightarrow \infty$. Now, if $\lim_{N\rightarrow\infty} p_{N1}\log t=0$, then with high probability
$$\lim_{N\rightarrow\infty}\lim_{t\rightarrow \infty}|T(S(G_m^t))|\sum_{i}p_{N i}^{3}\leq \lambda(1+\varepsilon){m+1\choose 2} \lim_{N\rightarrow\infty}\lim_{t\rightarrow \infty}p_{N1}\log t=0.$$
Therefore, the RHS goes to zero at $N\rightarrow \infty$ and by dominated convergence theorem, the number of monochromatic edges $W_{t}$ converges in distribution to $\dPois(\lambda)$. 

Thus, taking $t=\lfloor\frac{r}{s_{N}^2}\rfloor$, we get $$\lim_{N\rightarrow \infty}\P\left(s_{N}^2T_{N}^{PA(m)} > r\right)=\lim_{N\rightarrow \infty}\P(W_{ \lfloor\frac{r}{s_{N}^2}\rfloor}=0)=e^{-mr}, $$
completing the proof of Theorem \ref{th:palimitdistribution}.

\subsection{The Infinite Path} Define $\cT_{N, m}$ to be the first time there exists a monochromatic path of length $m$ in a sequential coloring of the infinite path $\cZ$ with the probability distribution $\P_N$. This problem can be re-formulated as follows: $\{X_{i}\}_{i\geq 1}$ be an i.i.d. $\P_N$ sequence, and  
\begin{equation}
\cT_{N, m}=\inf\{t\geq m: X_{t}=\ldots=X_{t-m+1}\}.
\label{eq:colorpath}
\end{equation}
Similar to the proof of Theorem \ref{th:palimitdistribution}, using the Stein method based on dependency graph, the limiting distribution of $\cT_{N, m}$ can be determined.


\begin{thm}Let $s_{N, m}:=\left(\sum_{i}p_{N i}^m\right)^\frac{1}{m}$ and suppose $\lim_{N\rightarrow \infty}p_{N1}=0$. Then for $r \geq 0$,
\begin{equation}
\lim_{N\rightarrow \infty}\P\left(s_{N, m}^m\cT_{N, m} > r\right) = e^{-r}.
\label{eq:pathlimitdistribution}
\end{equation}
\label{th:pathlimitdistribution}
\end{thm}

\begin{proof}For $t\geq m$, and $s\in[1, t-m+1]$, define
$$Z_{s, t}=\pmb 1\{X_{ s}=X_{s+1}=\cdots=X_{s+m-1}\},$$ and $W_{t}=\sum_{s=1}^{t-m+1} Z_{s, t}$. Note that $W_{t}$ counts the number of monochromatic paths with $m$ vertices in the path spanned by the vertices $\{1, 2, \ldots t\}$. Note that $\lambda_{t}:=\E(W_{t})=(t-m+1)\sum_{i}p_{N i}^m$ and
$$\P(Z_{s_1, t}=Z_{s_2, t}=1)=
\left\{
\begin{array}{cc}
\left(\sum_{i}p_{N i}^m\right)^2 &   \text{ if } |s_2-s_1|\geq m   \\
\sum_{i}p_{N i}^{m-|s_2-s_1|} &   \text{ otherwise.}   \\
\end{array}
\right.
$$
Since $Z_{s_1, t}$ and $Z_{s_2, t}$ are independent if $|s_2-s_1|\geq m$, the dependency graph associated with the variables $\{Z_{s, t}\}_{s\geq 1}$ has an edge between $Z_{s_1, t}$ and $Z_{s_2, t}$ if $|s_2-s_1|< m$. Then the error term in (\ref{th:steinsdependencygraph}) becomes
$$||W_{i}-\dPois(\lambda_{i})||_{\mathrm{TV}}\leq \frac{4m(t-m+1)\left(\sum_{b=0}^{m-1}\sum_{i}p_{N i}^{m-b}+\left(\sum_{i}p_{N i}^m\right)^2 \right)}{(t-m+1)\sum_{i}p_{N i}^m}.$$
The error term in the RHS goes to 0 if $t=t(N)$ is such that $\lambda_t=(t-m+1)\sum_{i}p_{N i}^m\rightarrow \lambda>0$, as $N\rightarrow\infty$. 

For $r > 0$, let $t=\ceil{\frac{r}{s_{N, m}^m}}$. Then $$\P\left(s_{N,m}^m\cT_{N, m} > r\right)=\P(W_{ \lfloor\frac{r}{s_{N, m}^m}\rfloor}=0)\rightarrow e^{-r},$$
completing the proof of the result.
\end{proof}

\small{\noindent{\bf Acknowledgement:} The author is indebted to his advisor Persi Diaconis for proposing the problem to him and for his inspirational guidance. The author thanks Sourav Chatterjee, Shirshendu Ganguly, Prasad Tetali, and Qingyuan Zhao for helpful discussions. The author also thanks an anonymous referee for carefully reading the manuscript and for providing many helpful comments, which improved the quality and presentation of the paper.}

\normalsize
\appendix
\section{Verifying Absolute Convergence}
\label{sec:appendix}

For every $a\in [q]$, let $\vec \psi_a=(\psi_{1, a}, \psi_{2, a}, \ldots)$ and $\sum_i(\sum_{a=1}^qc_a\psi_{i, a})^2<\infty$. Define 
\begin{equation}
g(r, \vec \psi_1, \vec \psi_2, \ldots, \vec \psi_q)=\sum_i \left\{-r\sum_{a=1}^qc_a\psi_{i, a}+ \log\left(1+\sum_{a=1}^q e^{rc_a\psi_{i, a}}-q\right)\right\}.
\label{eq:g}
\end{equation}
Also, let $Q_i(z)=\sum_{s=1}^\infty \gamma_{i, s}\frac{z^s}{s!}$ and $\gamma_{i, s}=\sum_{a=1}^q c_a^s \psi_{i, a}^s$. By standard rearrangement identities \cite{fsbook}, for $s\geq 1$, 
\begin{align}
\label{Qzk}
Q_i^s(z)=\sum_{\substack{j_1, j_2, \ldots, j_s\\j_t\geq 1, \forall t \in [s]}}z^{\sum_{b=1}^s j_b}\left(\prod_{b=1}^s \frac{\sum_{a=1}^q c_a^{j_b}\psi_{i, a}^{j_b}}{j_b!}\right).
\end{align}

\begin{lem}Let $\sum_i(\sum_{a=1}^qc_a\psi_{i, a})^2:=C<\infty$, and define $\psi_{\max}=\max_{a\in[q]}\max_{i \in [n]}\psi_{i, a}$. Then for $r \in R=\{r\in \R^+: \psi_{\max}r < \log(1+1/q)\}$,	
\begin{align}
\label{eq:gexpansion}
g(r, \vec \psi_1, \vec \psi_2, \ldots, \vec \psi_q)=&-\frac{r^2}{2}\sum_i\left(\left(\sum_{a}c_a\psi_{i, a}\right)^2-\sum_{a=1}^q c_a^2\psi_{i, a}^2\right)\nonumber\\
 &+A(r, \vec \psi_1, \vec \psi_2, \ldots, \vec \psi_q)+B(r, \vec \psi_1, \vec \psi_2, \ldots, \vec \psi_q),
\end{align}
where
\begin{align*}
A(r, & \vec \psi_1, \vec \psi_2, \ldots,  \vec \psi_q)\nonumber\\
=& \sum_{j_1\geq 3}\frac{r^{j_1}}{j_1!}\sum_i\sum_{a=1}^q c_a^{j_1}\psi_{i, a}^{j_1}-\frac{1}{2}\sum_{\substack{j_1\geq 1, j_2\geq 1,\\j_1+j_2\geq 3}}r^{\sum_{b=1}^2 j_b}\sum_i\left(\prod_{b=1}^2 \frac{\sum_{a=1}^q c_a^{j_b}\psi_{i, a}^{j_b}}{j_b!}\right),
\end{align*}
and
$$B(r, \vec \psi_1, \vec \psi_2, \ldots, \vec \psi_q)=\sum_{s=3}^\infty\frac{(-1)^{s+1}}{s}Q_i^s(z).$$
Moreover, the series~(\ref{eq:gexpansion}) is absolutely convergent for $r\in R$.
\label{lm:seriesq}
\end{lem} 

\begin{proof}For $r\in R$, $\psi_{i, a}<\frac{1}{r c_a}\log\left(1+1/q\right)$, for all $a\in [q]$ and $i \in \mathbb N$, and  $|\sum_{a=1}^q e^{rc_a\psi_{i, a}}-q|< 1$. Thus, using the expansion of $\log(1+z)$, for $|z|<1$, and the expansion of $e^{-z}$, for all $z\in \mathbb R$, 
\begin{align}\label{gr}
g(r, \vec \psi_1, \ldots, \vec \psi_q)=&\sum_i\left\{-r \sum_{a=1}^qc_a\psi_{i, a}+\sum_{s=1}^\infty\frac{(-1)^{s+1}}{s}\left(\sum_{a=1}^q e^{rc_a\psi_{i, a}}-q\right)^s\right\}\nonumber\\
=&\sum_i \left\{-r\sum_{a=1}^qc_a\psi_{i, a}+ \sum_{s=1}^\infty\frac{(-1)^{s+1}}{s}\left(\sum_{a=1}^q \sum_{t=1}^\infty\frac{r^tc_a^t\psi_{i, a}^t}{t!}\right)^s\right\}\nonumber\\
=&\sum_i \left\{-r\sum_{a=1}^qc_a\psi_{i, a}+ \sum_{s=1}^\infty\frac{(-1)^{s+1}}{s} Q_i(r)^s\right\},
\end{align}

Note that $Q_i(r)=r\sum_{a=1}^q c_a\psi_{i, a} + \sum_{j_1=2}^\infty \frac{r^{j_1}}{j_1!}\left( \sum_{a=1}^q c_a^{j_1}\psi_{i, a}^{j_1} \right)$. Thus, (\ref{gr}) implies, 
\begin{align}
g(r, & \vec \psi_1, \vec \psi_2, \ldots, \vec \psi_q)=\sum_{i}\left\{\sum_{j_1=2}^\infty \frac{r^{j_1}}{j_1!} \left( \sum_{a=1}^q c_a^{j_1}\psi_{i, a}^{j_1}\right)+\sum_{s=2}^\infty\frac{(-1)^{s+1}}{s} Q_i^s(r)\right\}.
\end{align}

Define $$\sS:=\sum_{i}\left\{\sum_{j_1=2}^\infty \frac{r^{j_1}}{j_1!} \left( \sum_{a=1}^q c_a^{j_1}\psi_{i, a}^{j_1}\right)+\sum_{s=2}^\infty\frac{Q_i^s(r)}{s} \right\}.$$
To show the series $g(r, \vec \psi_1, \vec \psi_2, \ldots, \vec \psi_q)$ is absolutely convergent for $r \in R$, it suffices to show $\sS< \infty$. To this end,
\begin{eqnarray}\label{s}
\sS&\leq&\sum_{j_1=2}^\infty r^{j_1}\sum_i\sum_{a=1}^q c_a^{j_1}\psi_{i, a}^{j_1}+\sum_{s=2}^\infty\sum_{i}Q_i^s(r).
\end{eqnarray}
Note that, for $r \in R$, 
\begin{align*}
\sum_{j_1= 2}^\infty r^{j_1}\sum_i\sum_{a=1}^q c_a^{j_1}\psi_{i, a}^{j_1} \leq \sum_{j_1= 2}^\infty r^{j_1}\sum_i\left(\sum_{a=1}^q c_a\psi_{i, a}\right)^{j_1} \leq & C\sum_{j_1= 2}^\infty r^{j_1}\psi_{\max}^{j_1-2}\nonumber\\
\leq & \frac{Cr^2}{(1-r{\psi_{\max}})}.
\end{align*} Similarly, recalling (\ref{Qzk}), it can shown that 
$\sum_{s=2}^\infty\sum_{i}Q_i^s(r)< \infty$, for $r\in R$.
This implies that the series in the RHS of~\eqref{s} is finite, and $g(r, \vec \psi_1, \ldots, \vec \psi_q)$ is absolutely convergent. 

The result now follows by interchanging the order of the summations and rearranging the terms.
\end{proof}

\end{document}